\documentclass[a4paper,10pt,leqno]{amsart}
\title{On the stable Cannon Conjecture}

\author{Steve Ferry}
            \email{sferry@math.rutgers.edu}
          \urladdr{http://www.math.rutgers.edu/~sferry/}
          \address{Department of Mathematics\\Rutgers University\\
 Hill Center, Busch Campus\\
 Piscataway, NJ 08854--8019, U.S.A.}

\author{Wolfgang L\"uck}
        \email{wolfgang.lueck@him.uni-bonn.de}
          \urladdr{http://www.him.uni-bonn.de/lueck}

        \address{Mathematisches Institut der Universit\"at Bonn\\
                Endenicher Allee 60\\
                53115 Bonn, Germany}
         
\author{Shmuel Weinberger}
            \email{shmuel@math.uchicago.edu}
          \urladdr{http://www.math.uchicago.edu/\%7Eshmuel/}
          \address{Department of Mathematics\\
                   University of Chicago\\
                   5734 S. University Avenue
                   Chicago, IL 60637--151, U.S.A.}

         \date{March, 2019}

\usepackage{hyperref}\usepackage{color}
\usepackage{pdfsync}
\usepackage{calc}
\usepackage{enumerate,amssymb,mathtools}
\usepackage[arrow,curve,matrix,tips,2cell]{xy}
  \SelectTips{eu}{10} \UseTips
  \UseAllTwocells

\DeclareMathAlphabet\EuR{U}{eur}{m}{n}
\SetMathAlphabet\EuR{bold}{U}{eur}{b}{n}

\makeindex             

\theoremstyle{plain}
\newtheorem{theorem}{Theorem}[section]
\newtheorem{lemma}[theorem]{Lemma}

\newtheorem{conjecture}[theorem]{Conjecture}

\theoremstyle{definition}

\newtheorem{definition}[theorem]{Definition}
\newtheorem{example}[theorem]{Example}

\newtheorem{remark}[theorem]{Remark}

{\catcode`@=11\global\let\c@equation=\c@theorem}

\newcommand{\comsquare}[8]                   
{\begin{CD}
#1 @>#2>> #3\\
@V{#4}VV @V{#5}VV\\
#6 @>#7>> #8
\end{CD}
}

\newcommand{\xycomsquare}[8]                   
{\xymatrix
{#1 \ar[r]^{#2} \ar[d]^{#4} &
#3 \ar[d]^{#5}  \\
#6\ar[r]^{#7} &
#8
}
}

\newcommand{\xycomsquareminus}[8]                      
{\xymatrix{#1 \ar[r]^-{#2} \ar[d]^-{#4} &
#3 \ar[d]^-{#5}  \\
#6\ar[r]^-{#7} &
#8
}
}

\newcommand{\caln}{{\mathcal N}}

\newcommand{\calv}{\mathcal{V}}
\newcommand{\calw}{\mathcal{W}}

\newcommand{\calfj}{{\mathcal F}\!{\mathcal J}}

\newcommand{\IQ}{{\mathbb Q}}
\newcommand{\IR}{{\mathbb R}}

\newcommand{\IZ}{{\mathbb Z}}

\newcommand{\bfE}{{\mathbf E}}
\newcommand{\bfe}{{\mathbf e}}

\newcommand{\bfi}{{\mathbf i}}

\newcommand{\bfL}{{\mathbf L}}

\newcommand{\bfpr}{{\mathbf {pr}}}

\newcommand{\bfsym}{{\mathbf {sym}}}

\newcommand{\curs}{\EuR}

\newcommand{\Or}{\curs{Or}}

\newcommand{\SPECTRA}{\curs{SPECTRA}}

\newcommand{\ENR}{\operatorname{ENR}}
\newcommand{\ANR}{\operatorname{ANR}}
\newcommand{\asmb}{\operatorname{asmb}}

\newcommand{\CAT}{\operatorname{CAT}}

\newcommand{\cent}{\operatorname{cent}}

\newcommand{\colim}{\operatorname{colim}}

\newcommand{\diam}{\operatorname{diam}}

\newcommand{\id}{\operatorname{id}}
\newcommand{\inti}{\operatorname{int}}

\newcommand{\im}{\operatorname{im}}

\newcommand{\pr}{\operatorname{pr}}

\newcommand{\tors}{\operatorname{tors}}

\newcommand{\Wh}{\operatorname{Wh}}
\newcommand{\sym}{\operatorname{sym}}

\newcommand{\twotor}{/2\text{-}\!\tors}

\newcommand{\pt}{\{\bullet\}}

\newcommand{\higherlim}[3]{{\setbox1=\hbox{\rm lim}
        \setbox2=\hbox to \wd1{\leftarrowfill} \ht2=0pt \dp2=-1pt
        \mathop{\vtop{\baselineskip=5pt\box1\box2}}
        _{#1}}^{#2}#3}

\newcommand{\version}[1]                       
{\begin{center} last edited on #1\\
last compiled on \today\\
name of texfile: \jobname
\end{center}
}

\newcounter{commentcounter}

\begin{document}

\typeout{---------------------------- flw_cannon.tex ----------------------------}


\typeout{------------------------------------ Abstract
  ----------------------------------------}

\begin{abstract} The Cannon Conjecture for a torsion-free hyperbolic group $G$ with
  boundary homeomorphic to $S^2$ says that $G$ is the fundamental group of an aspherical
  closed $3$-manifold $M$.  It is known that then $M$ is a hyperbolic $3$-manifold.  We
  prove the stable version that for any closed manifold $N$ of dimension greater or equal to $2$  there
  exists a closed manifold $M$ together with a simple homotopy equivalence
  $M \to N \times BG$. If $N$ is aspherical and $\pi_1(N)$ satisfies the
  Farrell-Jones Conjecture, then $M$ is unique up to homeomorphism.  
\end{abstract}

\maketitle


\typeout{------------------------------- Section 0: Introduction
  --------------------------------}

\setcounter{section}{-1}
\section{Introduction}


\subsection{The motivating conjectures by Wall and Cannon}
\label{subsec:The_motivating_conjectures_by_Wall_and_cannon}

This paper is motivated by the following two conjectures which will be reviewed in
Sections~\ref{sec:Short_review_of_Poincare_duality_groups} and~\ref{sec:Short_review_of_the_Cannon_Conjecture}. 

\begin{conjecture}[A Conjecture on Poincar\'e duality groups and closed aspherical 
  $3$-manifolds by Wall]\label{con:Poincare_duality_groups_of_dimension_three}
  Every Poincar\'e duality group of dimension $3$ is the fundamental group of an
  closed aspherical $3$-manifold.
\end{conjecture}

\begin{conjecture}[Cannon Conjecture in the torsion-free case]%
\label{con:Cannon_conjecture_in_the_torsionfree_case}
  Let $G$ be a torsion-free hyperbolic group. Suppose that its boundary is homeomorphic to
  $S^2$.

  Then $G$ is the fundamental group of a closed hyperbolic $3$-manifold.
\end{conjecture}

We will investigate whether these conjectures are true stably. More precisely, we ask
whether for any closed smooth manifold $N$ of dimension $\ge 2$ the product $BG \times N$
is simple homotopy equivalent to a closed smooth manifold.  Notice that
for a torsionfree hyperbolic group $G$ there is a finite $CW$-complex model for $BG$ by the Rips complex.
The Whitehead group of $G$ is known to be trivial, so the simple homotopy type of $BG$ is well-defined.
We will also consider the analogous questions in the in the PL and topological categories.


\subsection{The main results}
\label{subsec:The_main_results}

In the sequel $\underline{\IR^a}$ denotes the trivial $a$-dimensional vector bundle.

\begin{theorem}[Vanishing of the surgery obstruction]%
\label{the:Vanishing_of_the_surgery_obstruction}
  Let $G$ be a hyperbolic $3$-dimensional Poincar\'e duality group.

  Then there exist a closed smooth 3-manifold M and a normal map of degree one (in the sense of surgery theory)
  \[
    \xymatrix{TM \oplus \underline{\IR^a} \ar[r]^-{\overline{f}} \ar[d] & \xi \ar[d]
      \\
      M \ar[r]^-f & BG }
  \]
  satisfying

  \begin{enumerate}

  \item\label{the:Vanishing_of_the_surgery_obstruction:BG} The space BG is a finite
    $3$-dimensional $CW$-complex;

  \item\label{the:Vanishing_of_the_surgery_obstruction:homology} The map
    $H_n(f;\IZ) \colon H_n(M;\IZ) \xrightarrow{\cong} H_n(BG;\IZ)$ is bijective for all
    $n \ge 0$;

  \item\label{the:Vanishing_of_the_surgery_obstruction:surgery_obstruction} The simple
    algebraic surgery obstruction $\sigma(f,\overline{f}) \in L_3^s(\IZ G)$ vanishes.

  \end{enumerate}
\end{theorem}

Notice that the vanishing of the surgery obstruction does not imply that we can arrange by
surgery that $f$ is a simple homotopy equivalence since this works only in dimensions
$\ge 5$. In dimension $3$ we can achieve at least a $\mathbb Z G$-homology equivalence.
See \cite[Theorem 11.3A]{Freedman-Quinn(1990)}.  

However, if we cross the normal map with a closed manifold $N$ of dimension $\ge 2$, the
resulting normal map has also vanishing surgery obstruction by the product formula and
hence can be transformed by surgery into a simple homotopy equivalence. Thus
Theorem~\ref{the:Vanishing_of_the_surgery_obstruction} implies
assertion~\eqref{the:stable_Cannon_Conjecture:simple} of
Theorem~\ref{the:stable_Cannon_Conjecture} below; the proof of
assertion~\eqref{the:stable_Cannon_Conjecture:universal_covering} of
Theorem~\ref{the:stable_Cannon_Conjecture} below will require more work.

\begin{theorem}[Stable Cannon Conjecture]\label{the:stable_Cannon_Conjecture}
  Let $G$ be a hyperbolic $3$-dimensional Poincar\'e duality group.  Let $N$ be any
  smooth, PL or topological manifold respectively which is closed and whose dimension is
  $\ge 2$.

  Then there is a closed smooth, PL or topological manifold $M$ and a normal map
  of degree one
  \[
    \xymatrix{TM \oplus \underline{\IR^a} \ar[d]\ar[r]^-{\underline{f}} & \xi \times TN
      \ar[d]
      \\
      M\ar[r]^-f & BG \times N }
  \]
  satisfying

  \begin{enumerate}

  \item\label{the:stable_Cannon_Conjecture:simple} The map $f$ is a simple homotopy
    equivalence;

  \item\label{the:stable_Cannon_Conjecture:universal_covering} Let $\widehat{M} \to M$ be
    the $G$-covering associated to the composite of the isomorphism
    $\pi_1(f) \colon \pi_1(M) \xrightarrow{\cong} G \times \pi_1(N)$ with the projection
    $G \times \pi_1(N) \to G$. Suppose additionally that $N$ is aspherical, 
    $\dim(N) \ge 3$, and $\pi_1(N)$ is a Farrell-Jones group.

    Then $\widehat{M}$ is homeomorphic to $\IR^3 \times N$. Moreover, there is a compact topological
    manifold $\overline{\widehat{M}}$ whose interior is homeomorphic to $\widehat{M}$ and
    for which there exists a homeomorphism of pairs
    $(\overline{\widehat{M}},\partial \overline{\widehat{M}}) \to (D^3 \times N, S^2
    \times N)$.

  \end{enumerate}
\end{theorem}

We call a group $G$ a \emph{Farrell-Jones-group} if it satisfies the Full Farrell-Jones
Conjecture.  We will review what is known about the class of Farrell-Jones groups in
Theorem~\ref{the:status_of_the_Full_Farrell-Jones_Conjecture}.  For now, we mention that hyperbolic groups,  
CAT(0)-groups, and the fundamental groups of
 (not necessarily compact) $3$-manifolds (possibly with boundary) are Farrell-Jones
groups.

We have the following uniqueness statement.

\begin{theorem}[Borel Conjecture]\label{the:Borel} Let $M_0$ and $M_1$ be two 
  closed aspherical manifolds of dimension $n$ satisfying $\pi_1(M_{0}) \cong
  \pi_1(M_{1})$. Suppose one of the following conditions hold:

  \begin{itemize}

  \item We have $n \le 3$;
  \item We have $n = 4$ and $\pi_1(M_{0})$ is a Farrell-Jones group which is good in the
    sense of Freedman~\cite{Freedman(1983)};

  \item We have $n \ge 5$ and $\pi_1(M_{0})$ is a Farrell-Jones group.

  \end{itemize}

  Then any map $f \colon M_0 \to M_1$ inducing an isomorphism of fundamental groups 
  is homotopic to a homeomorphism.
\end{theorem}

\begin{proof} The Borel Conjecture is true obviously in dimension $n \le 1$.  The Borel
  Conjecture is true in dimension $2$ by the classification of closed manifolds of
  dimension $2$. It is true in dimension $3$ since Thurston's Geometrization Conjecture
  holds.  This follows from results of Waldhausen (see Hempel~\cite[Lemma~10.1 and
  Corollary~13.7]{Hempel(1976)}) and Turaev, see~\cite{Turaev(1988)}, as explained for
  instance in~\cite[Section~5]{Kreck-Lueck(2009nonasph)}.  A proof of Thurston's
  Geometrization Conjecture is given in~\cite{Kleiner-Lott(2008), Morgan-Tian(2014)}
  following ideas of Perelman. The Borel Conjecture follows from surgery theory in
  dimension $\ge 4$, see for instance~\cite[Proposition~0.3]{Bartels-Lueck(2012annals)}.
\end{proof}

One cannot replace homeomorphism by diffeomorphism in Theorem~\ref{the:Borel}.  The torus
$T^n$ for $ n \ge 5$ is a counterexample, see~\cite[15A]{Wall(1999)}.  Other
counterexamples involving negatively curved manifolds are constructed by
Farrell-Jones~\cite[Theorem~0.1]{Farrell-Jones(1989b)}.


\subsection{Acknowledgments}
\label{subsec:Acknowledgements}

The first author thanks the Jack \& Dorothy Byrne Foundation and “the University of Chicago” for support during numerous visits. 
The paper is financially supported by the ERC 
Advanced Grant ``KL2MG-interactions'' (no.
662400) of the second author granted by the European Research Council, 
and by the Cluster
of Excellence ``Hausdorff Center for Mathematics'' at Bonn.
The third author 
was partially supported by NSF grant 1510178.

We thank Michel Boileau for fruitful discussions and hints and the referee who 
wrote a very detailed and helpful report.

The paper is organized as follows:
\tableofcontents


\typeout{------------------------------- Section 1: Short review of Poincare duality
  groups --------------------------------}

\section{Short review of Poincar\'e duality groups}
\label{sec:Short_review_of_Poincare_duality_groups}

\begin{definition}[Poincar\'e duality group]\label{def:Poincare_duality_group}
  A \emph{Poincar\'e duality group $G$ of dimension $n$} is a group satisfying:

  \begin{itemize}

  \item $G$ is of type {FP}, i.e. $\IZ$ admits a finite resolution by finitely generated projective $\IZ G$-modules;

  \item $H^i(G;\IZ G) \cong
    \begin{cases}
      0 & i \not = n;
      \\
      \IZ & i = n.
    \end{cases}
    $
  \end{itemize}
\end{definition}


\subsection{Basic facts about Poincar\'e duality groups}
\label{subsec:Basic_facts_about_Poincare_duality_groups}

\begin{itemize}

\item A Poincar\'e duality group is finitely generated and torsion free;

\item For $n \ge 4$ there exist $n$-dimensional Poincar\'e duality groups which are not
  finitely presented, see~\cite[Theorem~C]{Davis(1989)};

\item If $G$ is a Poincar\'e duality group of dimension $n \ge 3$ then $BG$ is a finitely dominated
  $n$-dimensional Poincar\'e complex in the sense of Wall~\cite{Wall(1967)} if and only if
  $G$ is finitely presented, see~\cite[Theorem~1]{Johnson+Wall(1972)}. If $\widetilde{K}_0(\IZ G)$ vanishes, 
  then $BG$ is homotopy equivalent to
  finite $n$-dimensional $CW$-complex, see~\cite[Theorem~F]{Wall(1965a)};

\item If $G$ is the fundamental group of a closed aspherical manifold of dimension $n$,
  then $BG$ is homotopy equivalent to a finite $n$-dimensional $CW$-complex and in
  particular $G$ is finitely presented. In fact, every compact $\ENR$ of dimension $n>2$ is 
 homotopy equivalent to a finite $n$-dimensional polyhedron, see West~\cite{West(1977)};

\item To our knowledge there exists in the literature no example of a
  $3$-dimensional Poincar\'e duality group which is not homotopy equivalent to a finite
  $3$-dimensional $CW$-complex;

\item Every $2$-dimensional Poincar\'e duality group is the fundamental group of a closed
  surface.  This result is due to Bieri, Eckmann and Linnell, see for
  instance,~\cite{Eckmann(1987)}.

\end{itemize}


\subsection{Some prominent conjectures and results about Poincar\'e duality groups}
\label{subsec:Some_prominent_conjectures_and_results_about_Poincare_duality_groups}

\begin{conjecture}[Poincar\'e duality groups and closed aspherical manifolds]%
\label{con:Poincare_duality_groups_and_aspherical_closed_manifolds}
  Every finitely presented Poincar\'e duality group is the fundamental group of a closed  aspherical
  topological manifold.
\end{conjecture}

A weaker version is

\begin{conjecture}[Poincar\'e duality groups and closed aspherical $\ENR$ homology ma\-ni\-folds]%
\label{con:Poincare_duality_groups_and_aspherical_closed_homology_manifolds}
  Every finitely presented Poincar\'e duality group is the fundamental group of a 
  closed aspherical $\ENR$ homology manifold.
\end{conjecture}

Michel Boileau has informed us about the following two facts: 

\begin{theorem}\label{the:subgroup_suffices}
  A Poincar\'e duality group $G$ of dimension $3$ is the fundamental group of a 
  closed aspherical $3$-manifold if and only if $G$ contains a subgroup $H$, which is the
  fundamental group of a closed aspherical $3$-manifold.
\end{theorem}
\begin{proof}
Let $H$ be a subgroup of $G$ which is the
  fundamental group of an irreducible  closed $3$-manifold. Suppose that the index of $H$ in $G$ is
  infinite.  Then the cohomological dimension of $H$ is smaller than
  the cohomological dimension of $G$ by~\cite{Strebel(1977)}. 
  Since the cohomological dimension of both $H$ and $G$ is three, we
  get a contradiction.  Hence the index of $H$ in $G$ is finite. 
  The solution of Thurston's
Geometrization Conjecture by Perelman, see~\cite{Morgan-Tian(2014)}, implies that
$G$ is the fundamental group of an irreducible  closed $3$-manifold,  see for
instance~\cite[Theorem~5.1]{Groves-Manning-Wilton(2012)}. Since a closed  $3$-manifold is aspherical 
if and only if it is irreducible and has infinite fundamental group, Lemma~\ref{the:subgroup_suffices} follows.
\end{proof}

Moreover, Theorem~\ref{the:subgroup_suffices} and the works of
Cannon-Cooper~\cite{Cannon-Cooper(1992)},
Eskin-Fisher-Whyte~\cite{Eskin-Fisher-Whyte(2007)},
Kapovich-Leeb~\cite{Kapovich-Leeb(1997)}, and~Rieffel~\cite{Rieffel(2001)} imply

\begin{theorem}\label{the:quaisisometry_suffices}
  A Poincar\'e duality group $G$ of dimension $3$ is the fundamental group of a
  closed   aspherical $3$-manifold if and only if it is quasiisometric to the fundamental
  group of a closed aspherical $3$-manifold.
\end{theorem}

The next result is due to Bowditch~\cite[Corollary~0.5]{Bowditch(2004)}.

\begin{theorem}\label{the:center}
  If a Poincar\'e duality group of dimension $3$ contains an infinite normal cyclic
  subgroup, then it is the fundamental group of a closed Seifert $3$-manifold.
\end{theorem}

The following result follows from the algebraic torus theorem of
Dunwoody-Swenson~\cite{Dunwoody-Swenson(2000)}.

\begin{theorem}\label{the:atoroidal} Let $G$ be a $3$-dimensional Poincar\'e duality
  group. Then precisely one of the following statements are true:

  \begin{enumerate}

  \item It is the fundamental group of a closed Seifert $3$-manifold;

  \item It splits as an amalgam or HNN extension over a subgroup $\IZ \oplus \IZ$;

  \item It is atoroidal, i.e., it contains no subgroup isomorphic to $\IZ \oplus \IZ$.

  \end{enumerate}
\end{theorem}

\begin{conjecture}[Weak hyperbolization Conjecture]
  An atoroidal $3$-dimensional Poincar\'e duality group is hyperbolic.
\end{conjecture}

The next result is due to Kapovich-Kleiner~\cite[Theorem~2]{Kapovich-Kleiner(2007)}.

\begin{theorem}
  A $3$-dimensional Poincar\'e duality group which is a CAT(0)-group and atoroidal is
  hyperbolic.
\end{theorem}

We conclude from~\cite[Theorem~2.8 and Remark~2.9]{Bestvina(1996)}.

\begin{theorem}\label{the_hyperbolic_boundary_sphere} Let $G$ be a hyperbolic
  $3$-dimensional Poincar\'e duality group. Then its boundary is homeomorphic to $S^2$.
\end{theorem}


\subsection{High-dimensions}

\begin{theorem}[Poincar\'e duality groups and $\ENR$ homology manifolds]%
\label{the:FJ_and_Borel-existence}
  Let $G$ be a finitely presented torsion-free group which is a Farrell-Jones group.
  \begin{enumerate}

  \item\label{the:FJ_and_Borel-existence:ex} Then for $n \geq 6$ the following are
    equivalent:
    \begin{enumerate}
    \item\label{the:FJ_and_Borel:duality} $G$ is a Poincar\'e duality group of dimension
      $n$;
    \item\label{the:FJ_and_Borel:homology-mfd} There exists a closed $\ENR$ homology
      manifold $M$ homotopy equivalent to $BG$.  In particular, $M$ is aspherical and
      $\pi_1(M) \cong G$;
    \end{enumerate}
  \item\label{the:FJ_and_Borel-existence:DDP} If the statements in
    assertion~\eqref{the:FJ_and_Borel-existence:ex} hold, then the closed $\ENR$ homology
    manifold $M$ appearing there can be arranged to have the DDP, see Definition~\ref{def:DDP};
  \item\label{the:FJ_and_Borel-existence:uniqueness} If the statements in
    assertion~\eqref{the:FJ_and_Borel-existence:ex} hold, then the closed  $\ENR$ homology
   manifold $M$ appearing there is unique up to $s$-cobordism of $\ENR$ homology
    manifolds;

  \end{enumerate}
\end{theorem}
\begin{proof} See
Bartels-L\"uck-Weinberger~\cite[Theorem~1.2]{Bartels-Lueck-Weinberger(2010)}. It relies strongly on 
the surgery theory for $\ENR$ homology manifolds,  see for instance~\cite{Bryant-Ferry-Mio-Weinberger(1996), Ferry(2010),
Pedersen-Quinn-Ranicki(2003)}.
\end{proof}

The question whether a closed $\ENR$ homology manifold, which has dimension $\ge 5$ and has
the DDP, is a topological manifold is decided by Quinn's obstruction, see
Section~\ref{sec:Short_review_of_Quinns_obstruction}.

More information about Poincar\'e duality groups can be found for
instance~\cite{Davis(2000Poin)} and~\cite{Wall(2004)}.


\typeout{------------------------------- Section 2: Short review of the Cannon Conjecture
  --------------------------------}

\section{Short review of the Cannon Conjecture}
\label{sec:Short_review_of_the_Cannon_Conjecture}

The following conjecture is taken from~\cite[Conjecture~5.1]{Cannon-Swenson(1998)}.

\begin{conjecture}[Cannon Conjecture]\label{con:Cannon_conjecture} Let $G$ be a
  hyperbolic group. Suppose that its boundary is homeomorphic to $S^2$.

  Then $G$ acts properly cocompactly and isometrically on the $3$-dimensional hyperbolic
  space.
\end{conjecture}

If $G$ is torsion free, then the Cannon Conjecture~\ref{con:Cannon_conjecture} reduces to
the Cannon Conjecture for torsion free groups~\ref
{con:Cannon_conjecture_in_the_torsionfree_case}.

\begin{remark}\label{rem:Thurston_does_not_imply_Cannon} We mention that
  Conjecture~\ref{con:Cannon_conjecture_in_the_torsionfree_case} is open and does not
  follow from Thurston's Geometrization Conjecture which is known to be true by the work
  of Perelman, see~Morgan-Tian~\cite{Morgan-Tian(2014)}.
\end{remark}

The next result is due to Bestvina-Mess~\cite[Theorem~4.1]{Bestvina-Mess(1991)} and says
that for the Cannon Conjecture one just has to find some  closed aspherical $3$-manifold with $G$ as
fundamental group.

\begin{theorem}\label{the:manifold:realization_is_enough} Let $G$ be a hyperbolic group
  which is the fundamental group of a closed aspherical $3$-manifold $M$.

  Then the universal covering $\widetilde{M}$ of $M$ is homeomorphic to $\IR^3$ and its
  compactification by $\partial G$ is homeomorphic to $D^3$, and the Geometrization
  Conjecture of Thurston implies that $M$ is hyperbolic and $G$ satisfies the Cannon
  Conjecture~\ref{con:Cannon_conjecture_in_the_torsionfree_case}.
\end{theorem}

Ursula Hamenst\"adt informed us that she has a proof for the following result.

\begin{theorem}[Hamenst\"adt]\label{the:Hamenstaedt}
  Let $G$ be a hyperbolic group $G$ whose boundary is homeomorphic to $S^{n-1}$.

  Then $G$ acts properly and cocompactly on $S^{n-1} \times \IR^n$.
\end{theorem}

Hamenst\"adt's result is proved by completely different methods and does not need the
assumption that $G$ is torsion free.  It aims for $n = 3$ at construction of the sphere tangent bundle of the 
universal covering  of the conjectured hyperbolic $3$-manifold $M$ appearing in the Cannon
Conjecture~\ref{con:Cannon_conjecture}, where we aim at constructing $M$ for $BG \times N$ for any
closed  manifold $N$ with $\dim(N) \ge 2$.


\subsection{The high-dimensional analogue of the Cannon Conjecture}

The following result is taken from~\cite[Theorem~A]{Bartels-Lueck-Weinberger(2010)}.

\begin{theorem}[High-dimensional Cannon Conjecture]%
\label{the:High-dimensional_Cannon_Conjecture}
  Let $G$ be a torsion free hyperbolic group and let $n$ be an integer $\geq 6$. The
  following statements are equivalent:
  \begin{enumerate}
  \item\label{the:High-dimensional_Cannon_Conjecture:sphere} The boundary $\partial G$ is
    homeomorphic to $S^{n-1}$;

  \item\label{the:High-dimensional_Cannon_Conjecture:manifold} There is a 
    closed aspherical topological manifold $M$ such that $G \cong \pi_1(M)$, its universal
    covering $\widetilde{M}$ is homeomorphic to $\IR^n$ and the compactification of
    $\widetilde{M}$ by $\partial G$ is homeomorphic to $D^n$;

  \end{enumerate}

  Moreover, the aspherical manifold $M$ appearing in
  assertion~\eqref{the:High-dimensional_Cannon_Conjecture:manifold} is unique up to
  homeomorphism.
\end{theorem}

In high dimensions there are exotic examples of hyperbolic $n$-dimensional Poin\-ca\-r\'e
duality groups $G$, see~\cite[Section~5]{Bartels-Lueck-Weinberger(2010)}. For instance,
for any integer $k \ge 2$ there are examples satisfying $\partial G = S^{4k+1}$ such that
$G$ is the fundamental group of a closed aspherical topological manifold, but not of an
closed aspherical smooth manifold.  For $n \ge 6$ there exists a closed aspherical
topological manifold whose fundamental group is hyperbolic but which cannot be
triangulated, see~\cite[page~800]{Davis-Fowler-Lafont(2014)}.

We mention without  giving the details that using the method of this paper one can 
prove  Theorem~\ref{the:High-dimensional_Cannon_Conjecture} also in the case $n = 5$.


\subsection{The Cannon Conjecture~\ref{con:Cannon_conjecture_in_the_torsionfree_case} in
  the torsion free case implies Theorem~\ref{the:Vanishing_of_the_surgery_obstruction} and
  Theorem~\ref{the:stable_Cannon_Conjecture}}
\label{subsec:Cannon_implies_main_Theorem}

Let $G$ be hyperbolic $3$-dimensional Poincar\'e duality group.  We want to show that then
all claims in Theorem~\ref{the:Vanishing_of_the_surgery_obstruction} and
Theorem~\ref{the:stable_Cannon_Conjecture} are obviously true, provided that the Cannon
Conjecture~\ref{con:Cannon_conjecture_in_the_torsionfree_case} in the torsion free case
holds for $G$.

We know already that there is a $3$-dimensional finite model\footnote{Since $G$ is hyperbolic, $\widetilde{K}_0(\IZ G)$ vanishes.  
The rest follows from Subsection1.1 above.} for $BG$  and $\partial G$ is
$S^2$.  By the Cannon Conjecture~\ref{con:Cannon_conjecture_in_the_torsionfree_case} we
can find a closed hyperbolic $3$-manifold $M$ together with a homotopy equivalence
$f \colon M \to BG$. Since $G$ is a Farrell-Jones group, $f$ is a simple homotopy
equivalence. We obviously can cover $f$ by a bundle map
$\overline{f} \colon TM \to \xi$ if we take $\xi$ to be $(f^{-1})^*TM$ for some homotopy
inverse $f^{-1} \colon BG \to M$ of $f$. Hence we get
Theorem~\ref{the:Vanishing_of_the_surgery_obstruction} and
assertion~\eqref{the:stable_Cannon_Conjecture:simple} of
Theorem~\ref{the:stable_Cannon_Conjecture}. It remains to prove
assertion~\eqref{the:stable_Cannon_Conjecture:universal_covering} of
Theorem~\ref{the:stable_Cannon_Conjecture}.

The universal covering $\widetilde{M}$ is the hyperbolic $3$-space. Hence it is
homeomorphic to $\IR^3$ and the compactification
$\overline{\widetilde{M}} = \widetilde{M} \cup \partial G$ is homeomorphic to $D^3$.  In
particular $\overline{\widetilde{M}}$ is a compact manifold whose interior is
$\widetilde{M}$ and whose boundary is $S^2$.  Hence $\overline{\widetilde{M}} \times N$ is
a compact manifold and there is a homeomorphism
$(\overline{\widetilde{M}} \times N, \partial(\overline{\widetilde{M}} \times N))
\xrightarrow{\cong} (D^3 \times N, S^2 \times N)$.


\subsection{When does the Cannon
  Conjecture~\ref{con:Cannon_conjecture_in_the_torsionfree_case} in the torsion free case
  follow from Theorem~\ref{the:stable_Cannon_Conjecture}}
\label{subsec:Cannon_from_main_Theorem}

Next we discuss what would be needed to conclude the Cannon
Conjecture~\ref{con:Cannon_conjecture_in_the_torsionfree_case} in the torsion free case
from Theorem~\ref{the:stable_Cannon_Conjecture}.

Let $G$ be a hyperbolic group such that $\partial G$ is $S^2$. Then $G$ is a
$3$-dimensional Poincar\'e duality group by
Bestvina-Mess~\cite[Corollary~1.3]{Bestvina-Mess(1991)}.   Fix any
closed aspherical manifold $N$ of dimension $\ge 2$ such that $\pi_1(N)$ is a
Farrell-Jones group.

We get from Theorem~\ref{the:stable_Cannon_Conjecture} a closed aspherical
$(3+\dim(N))$-dimensional manifold $M$ together with a homotopy equivalence
$f \colon M\to BG \times N$.  Let
$\alpha \colon \pi_1(M) \xrightarrow{\cong} G \times \pi_1(N)$ be the isomorphism
$\pi_1(f)$.  If $M'$ is any other closed aspherical manifold together with an isomorphism
$\alpha' \colon \pi_1(M') \xrightarrow{\cong} G \times \pi_1(N)$, then we conclude from
Theorem~\ref{the:status_of_the_Full_Farrell-Jones_Conjecture}~%
\eqref{the:status_of_the_Full_Farrell-Jones_Conjecture:Classes_of_groups:hyperbolic_groups}
 and~\eqref{the:status_of_the_Full_Farrell-Jones_Conjecture:inheritance:Passing_to_finite_direct_products}
that $\pi_1(M) \cong G \times \pi_1(N)$ is a Farrell-Jones group and  from
Theorem~\ref{the:Borel} that there exists a homeomorphism $u \colon M \to M'$ such that
$\alpha' \circ \pi_1(u)$ and $\alpha$ agree (up to inner automorphisms).  Hence the pair
$(M,\alpha)$ is unique and thus an invariant depending on $G$ and $N$ only.

What does the Cannon Conjecture~\ref{con:Cannon_conjecture_in_the_torsionfree_case} tell us about $(M,\alpha)$
and what do we need to know about $(M,\alpha)$ in order to prove the Cannon
Conjecture~\ref{con:Cannon_conjecture_in_the_torsionfree_case}? This is answered by the
next result.

\begin{lemma}\label{lem:comparing_general_N} Assume Theorem~\ref{the:stable_Cannon_Conjecture}
for a given $(G,N)$ and consider the above unique $(M,\alpha)$.  The following statements are equivalent

  \begin{enumerate}

  \item\label{lem:comparing_general_N:(1)} The Cannon
  Conjecture~\ref{con:Cannon_conjecture_in_the_torsionfree_case} holds for $G$;

  \item\label{lem:comparing_general_N:(2)}
  There is a closed $3$-manifold $M'$ and a homeomorphism
  $h \colon M \xrightarrow{\cong} M'\times N$ such that for the projection
  $p \colon M' \times N \to N$ the map $\pi_1(p \circ h)$
  agrees with the composite
  $\pi_1(M) \xrightarrow{\alpha} G \times \pi_1(N) \xrightarrow{\pr} \pi_1(N)$ for $\pr$
  the projection;

  \item\label{lem:comparing_general_N:(3)}
  There is a closed $3$-manifold $M'$ and a map
  $p \colon M \to N$ with  homotopy fiber $M'$ such that $\pi_1(p)$
  agrees with the composite
  $\pi_1(M) \xrightarrow{\alpha} G \times \pi_1(N) \xrightarrow{\pr} \pi_1(N)$ for $\pr$
  the projection.

\end{enumerate}
\end{lemma}
\begin{proof}~\eqref{lem:comparing_general_N:(1)} $\implies$~\eqref{lem:comparing_general_N:(2)}.
By the Cannon  Conjecture~\ref{con:Cannon_conjecture_in_the_torsionfree_case} there exists
a closed hyperbolic $3$-manifold $M'$ with $\pi(M') = G$.
Since $M'$ models $BG$, we can find a homotopy equivalence  $h \colon M \to M' \times N$ with
$\pi_1(h) = \alpha$. By Theorem~\ref{the:Borel} we can assume that $h$ is a homeomorphism.
\\[1mm]~\eqref{lem:comparing_general_N:(2)} $\implies$~\eqref{lem:comparing_general_N:(3)}
This is obvious.
\\[1mm]~\eqref{lem:comparing_general_N:(3)} $\implies$~\eqref{lem:comparing_general_N:(1)}
The long exact homotopy sequence associated to $p$  implies that $\pi_1(M') \cong G$
and $M'$ is aspherical. We conclude from Theorem~\ref{the:manifold:realization_is_enough}
that $M'$ is a closed hyperbolic $3$-manifold. Hence $G$ satisfies the Cannon
Conjecture~\ref{con:Cannon_conjecture_in_the_torsionfree_case}.
\end{proof}


\subsection{The special case $N = T^k$}
\label{subsec:The_special_case_N-Is_Tk}

Now suppose that in the situation of Subsection~\ref{subsec:Cannon_from_main_Theorem} we
take $N = T^k$ for some $k \ge 2$. Then we get a criterion, where $\alpha$ does not appear
anymore.

\begin{lemma}\label{lem:comparing_T_upper_k} Fix an integer $k \ge 2$.
Let $M$ be a closed aspherical $(3+k)$-dimensional manifold with fundamental group 
$G \times \IZ^k$, where $G$ is hyperbolic with $\partial G = S^{2}$.
Then the  following statements are equivalent:
\begin{enumerate}

  \item\label{lem:comparing_T_upper_k:(1)} The Cannon
  Conjecture~\ref{con:Cannon_conjecture_in_the_torsionfree_case} holds for $G$;

  \item\label{lem:comparing_T_upper_k:(2)}
  There is closed $3$-manifold $M'$ together with  a homeomorphism
  $h \colon M \xrightarrow{\cong} M'\times T^k$;

  \item\label{lem:comparing_T_upper_k:(3)}
  There is a closed $3$-manifold $M'$ and a map
  $p \colon M \to T^k$ with  homotopy fiber $M'$.

\end{enumerate}
\end{lemma}
\begin{proof}~\eqref{lem:comparing_T_upper_k:(1)}
  $\implies $~\eqref{lem:comparing_T_upper_k:(2)} This follows from
  Theorem~\ref{lem:comparing_general_N}.
  \\[1mm]~\eqref{lem:comparing_T_upper_k:(2)}
  $\implies $~\eqref{lem:comparing_T_upper_k:(3)} This is obvious.
  \\[1mm]~\eqref{lem:comparing_T_upper_k:(3)}
  $\implies $~\eqref{lem:comparing_T_upper_k:(1)} First we explain
  that we can assume that $\pi_1(p) \colon \pi_1(M) \to \pi_1(T^k)$ is
  surjective.  Since $M'$ is compact and has only finitely many path
  components, we conclude from the exact long homotopy sequence that the
  image of $\pi_1(p) \colon \pi_1(M) \to \pi_1(T^k)$ has finite index.
  Let $q \colon T^k \to T^k$ be a finite covering such that the image
  of $\pi_1(p)$ and $\pi_1(q)$ agree.  Then we can lift
  $p \colon M \to T^k$ to a map $p'\colon M \to T^k$ such that
  $q \circ p' = p$.  One easily checks that 
  that $\pi_1(p')$ is surjective and the homotopy fiber of $p'$ fiber is a finite covering of
  $M'$ and in particular a closed $3$-manifold. Hence we assume without loss
  of generality that $\pi_1(p)$ is surjective, otherwise replace $p$
  by $p'$.

  Let $K$ be the kernel of the map
  $\pi_1(p) \colon \pi_1(M) \cong G \times \IZ^k \to \pi_1(T^k) \cong
  \IZ^k$. Since $M$ and  $T^k$ are aspherical, the homotopy fiber of $p$ is homotopy equivalent to $B\!K$.
  Hence $K$ is the fundamental group of the closed aspherical $3$-manifold $M'$.
  Define $K' := K \cap \{1\} \times \IZ^k$. This is a normal
  subgroup of both $K$ and $\IZ^k$ if we identify
  $ \{1\} \times \IZ^k = \IZ^k$.

  We begin with the case, where $K'$ is trivial. Then the projection
  $\pr \colon G \times \IZ^k \to G$ induces an isomorphism
  $K \xrightarrow{\cong} L$ for $L = \pr(K) \subseteq G$.  
  We  conclude from Theorem~\ref{the:subgroup_suffices} that $G$ is the
  fundamental group of a closed
  $3$-manifold. Theorem~\ref{the:manifold:realization_is_enough}
  implies that $G$ is the fundamental group of a closed hyperbolic
  $3$-manifold.

  Next we consider the case where $K'$  is non-trivial.  Consider
  the following commutative diagram
  \[
    \xymatrix{ & \{1\} \ar[d] & \{1\} \ar[d] & \{1\} \ar[d] &
      \\
      \{1\} \ar[r] & K' \ar[r] \ar[d] & K \ar[r] \ar[d] & K/K' \ar[d]
      \ar[r] & \{1\}
      \\
      \{1\} \ar[r] & \IZ^k \ar[r] \ar[d] & G \times \IZ^k \ar[r]
      \ar[d] & G \ar[d] \ar[r] & \{1\}
      \\
      \{1\} \ar[r] & \IZ^k/K' \ar[r] \ar[d] & \IZ^k \ar[r] \ar[d] & Q
      \ar[r] \ar[d] & \{1\}
      \\
      & \{1\} & \{1\} & \{1\} & }
  \]
  where the upper and the middle rows and the left and the middle columns are the obvious
  exact sequences, and the map $\IZ^k/K' \to \IZ^k$ is the map making the diagram commutative.
  The group $Q$ is defined to be the cokernel of the map $\IZ^k/K' \to \IZ^k$, and all other arrows
  are uniquely determined by the property that the diagram commutes.  The so called nine-lemma, 
  which can be proved by an easy diagram
  chase, shows that all rows and columns are exact.

  Since $K' \subseteq \IZ^k \times \{1\} \subseteq \IZ^k \times G$,
  we have $K' \subseteq \cent(G \times \IZ^k)$. Since $K' \subseteq K \subseteq G \times \IZ^k$, 
  we conclude  $K' \subseteq \cent(K)$. Since $K$ is torsionfree and $K'$ is non-trivial, the 
  center of $K$ contains a copy of $\IZ$. We conclude from Theorem~\ref{the:center} that there is a closed aspherical
   Seifert $3$-manifold $S$ such that $K =\pi_1(N)$.  There exists a finite covering
  $\overline{S} \to S$ such that $\overline{S}$ is orientable, there is a principal
  $S^1$-fiber bundle $S^1 \to \overline{S} \to F_g$ for a closed orientable surface of
  genus $g \ge 1$, see~\cite[page~436 and Theorem~2.3]{Scott(1983)}. We obtain a short
  exact sequence $\{1\} \to \pi_1(S^1) \to \pi_1(\overline{S}) \to \pi_1(F_g) \to \{1\}$.
  The center of $\pi_1(\overline{S})$ contains the image of
  $\pi_1(S^1) \to \pi_1(\overline{S^1})$ since  we are considering a principal
  $S^1$-fiber bundle $S^1 \to \overline{S} \to F_g$ and the fiber transport is by self-homotopy equivalences
  of $S^1$ which are all homotopic to the identity. The center cannot be larger if $g \ge 2$ since
  $\cent(\pi_1(F_g))$ is trivial for $g \ge 2$.  If the center is larger and $g = 1$, the
  extension has to be trivial, after possibly passing to a finite covering of
  $\overline{S}$.  Hence we can arrange that there is a subgroup
  $\overline{K} \subseteq K$ of finite index such that $\cent(\overline{K}) \cong \IZ$ and
  $\overline{K}/\cent(\overline{K}) \cong \pi_1(F_g)$ holds for some $g \ge 1$, or we have
  $\overline{K} \cong \IZ^3$;  just take $\overline{K} = \pi_1(\overline{S})$.

  Next we show that $\cent(\overline{K})$ must be infinite cyclic. If
  $\cent(\overline{K})$ is not infinite cyclic, then $\overline{K}$ has to be $\IZ^3$.  We
  conclude that $K$ and hence also $K/K'$ are virtually finitely generated abelian.  Since
  $Q$ is abelian, we have the exact sequence $1 \to K/K' \to G \to Q \to 1$ and $G$ has
  cohomological dimension $3$, the group $G$ cannot be hyperbolic, a contradiction.  Hence
  $\cent(\overline{K})$ must be infinite cyclic and and
  $\overline{K}/\cent(\overline{K}) \cong \pi_1(F_g)$ for some $g \ge 1$.

  We have $\{0\} \not= K' \subseteq \cent(K)$ and
  $\cent(K) \cap \overline{K} \subseteq \cent (\overline{K}) \cong \IZ$.  Since $K'$ is
  torsion free and $[K:\overline{K}]$ is finite, $\cent(K)$ is a non-trivial torsion-free
  virtually cyclic group and hence $\cent(K)$ is infinite cyclic. Since $\cent(K)/K'$ is a
  finite subgroup of $K/K'$ and $K/K'$ is isomorphic to a subgroup of the torsion free
  group $G$, we have $K' = \cent(K)$.  The group
  $\overline{K}/(\overline{K} \cap \cent(K))$ is a subgroup of $K/K' = K/\cent(K)$ of
  finite index and admits an epimorphism onto
  $\overline{K}/\cent(\overline{K}) \cong \pi_1(F_g)$ whose kernel
  $\cent(\overline{K})/(\overline{K} \cap \cent(K))$ is finite. Since
  $\overline{K}/(\overline{K} \cap \cent(K))$ is isomorphic to a subgroup of the
  torsion free group $G$, this kernel is trivial and hence
  $\overline{K}/(\overline{K} \cap \cent(K)) \cong \pi_1(F_g)$.
  
  Since $K'$ is infinite cyclic, $Q$ contains a copy of $\IZ$ of finite index. Hence we
  can find a subgroup $G'$ of $G$ of finite index together with a short exact sequence
  $\{1\} \to K/K' \to G' \to \IZ \to \{1\}$. So there exists an automorphism
  $\phi \colon K/K' \to K/K'$ such that $G'$ is isomorphic to the semi-direct product
  $K/K' \rtimes_{\phi} \IZ$.  If we put $L := \overline{K}/(\overline{K} \cap \cent(K))$,
  then $L \cong \pi_1(F_g)$ and $L$ is a subgroup of finite index of the finitely generated 
  group $K'/K$.  Then $L' = \bigcap_{n \in \IZ} \phi^n(L)$ is a subgroup of $K/K'$ of
  finite index again which satisfies $L' \subseteq L$ and  $\phi(L') = L'$ 
  and for which there is an isomorphism
  $u \colon L' \xrightarrow{\cong} \pi_1(F_{g'})$ for some $g' \ge 1$. Let
  $\phi' \colon L' \to L'$ be the automorphism induced by $\phi$.  Then
  $G'' := L' \rtimes_{\phi'} \IZ$ is isomorphic to a subgroup of $G$ of finite index in
  $G$.  Choose a homeomorphism $h' \colon F_{g'} \to F_{g'}$ satisfying
  $\pi_1(h') = u \circ \phi' \circ u^{-1}$.  The mapping torus
  $T_{h'}$ is a closed aspherical $3$-manifold with $\pi_1(T_{h'}) \cong G''$.
  Theorem~\ref{the:subgroup_suffices} shows that $G$ is the fundamental group of a closed
  $3$-manifold. Theorem~\ref{the:manifold:realization_is_enough} implies that $G$ is the
  fundamental group of a closed hyperbolic $3$-manifold.
\end{proof}

\begin{remark}[manifold approximate fibration]\label{rem:MAF} 
  Some evidence for Lemma~\ref{lem:comparing_T_upper_k}  comes from the conclusion
  of~\cite[Theorem~1.8]{Farrell-Lueck-Steimle(2018)} that one can find for any epimorphism
  $\alpha \colon \pi_1(M) \to \pi_1(T^k)$ at least a 
  manifold approximate fibration $p \colon M \to T^k$ such that $\pi_1(p) = \alpha$.
\end{remark}


\section{The existence of a normal map of degree one}
\label{sec:The_existence_of_a_normal_map_of_degree_one}

We call a connected finite Poincar\'e complex $X$ \emph{oriented} if we have chosen a generator
$[X]$ of the infinite cyclic group $H_n^{\pi_1(X)}(\widetilde{X};\IZ^{w_1(X)})$. 
Notice that we do allow non-trivial $w_1(X)$. In this section we show

\begin{theorem}[Existence of a normal map]\label{the:existence_of_a_normal_map}
  Let $X$ be a connected  finite $3$-dimensional Poincar\'e complex.  Then there
exist an integer $a \ge 0$ and a vector bundle $\xi$ over $X$ and a normal map of degree
  one
  \[
    \xymatrix{TM \oplus \underline{\IR^a} \ar[r]^-{\overline{f}} \ar[d] & \xi
      \ar[d]
      \\
      M \ar[r]^-f & X }
  \]
\end{theorem}
\begin{proof}
  Any element $c \in H^k(BO;\IZ/2)$ determines up to homotopy a unique  map
  $\widehat{c} \colon BSG \to K(\IZ/2,k)$. It is characterized by the property that
  $c = H^k(\widehat{c};\IZ/2)(\iota_k)$ for the canonical element
  $\iota_k \in H^k(K(\IZ/2,k);\IZ/2)$ which corresponds to $\id_{\IZ/2}$ under the
  isomorphism
  \begin{multline*}
    H^k(K(\IZ/2,k);\IZ/2) \cong \hom_{\IZ}(H_k(K(\IZ/2,k);\IZ),\IZ/2)
    \\
    \cong \hom_{\IZ}(\pi_k(K(\IZ/2,k)),\IZ/2) \cong \hom_{\IZ}(\IZ/2,\IZ/2).
  \end{multline*}
  Next, we claim that the product of the maps given by the universal first and second Stiefel-Whitney
  classes $w_1 \in H^1(BO;\IZ/2)$ and $w_2 \in H^2(BO;\IZ/2)$
  \begin{equation}
    \widehat{w_1} \times \widehat{w_2} \colon BO  \to   K(\IZ/2,1) \times K(\IZ/2,2) 
   \label{widehat(w_1)_times_widehat(w_2)}
  \end{equation}
  is $4$-connected. Since $BO$ is connected, $\pi_1(BO) \cong \pi_2(BSO) = \IZ/2$ and
  $\pi_3(BSO) = 0$, it suffices to show that
  $\pi_k(\widehat{w_k}) \colon \pi_k(BO) \to \pi_k(K(\IZ/2,k))$ is non-trivial for
  $k =1,2$.  This is easily proved using the fact the Hopf fibration $S^1 \to S^3 \to S^2$
  has non-trivial second Stiefel-Whitney class.  Hence for any $3$-dimensional complex $X$
  stable vector bundles over $X$ are classified by $w_1$ and $w_2$. The map
  induced by composition with the map~\eqref{widehat(w_1)_times_widehat(w_2)}
  \[ [X,BSO] \to [X,K(\IZ/2,1) \times K(\IZ/2,2)] = H^1(X;\IZ/2) \times H^2(X;\IZ/2),
  \]
  is bijective.  For a vector bundle $\xi$ with classifying map $f_{\xi}$ the
  class $[f_{\xi}]$ goes to $(w_1(\xi),w_2(\xi))$.

  We conclude from~\cite[page~44]{Hausmann-Vogel(1993)} that there is a closed manifold
  $M$ together with a map $f \colon M \to X$ such that $w_1(M) = f^*w_1(X)$ and the
  induced map
  $H_3^{\pi_1(M)}(\widetilde{M};\IZ^{w_1(M)}) \xrightarrow{\cong}
  H_3^{\pi_1(X)}(\widetilde{X};\IZ^{w_1(X)})$ is an isomorphism of infinite cyclic groups.
  The proof in the general case is a variation of the one for trivial $w_1(X)$ which we sketch
  next. The Atiyah-Hirzebruch spectral sequence applied to the homology theory
  $\Omega_*$ given by oriented bordism yields an epimorphism
  \[
    \Omega_3(X) \xrightarrow{\cong} H_3(X;\IZ), \quad [f \colon M \to BG] \mapsto
    f_*([M])
  \]
  since the projection $X \to \pt$ induces an epimorphism $\Omega_3(X) \to \Omega_3(\pt)$
  and there is a map of degree one from $X$ to $S^3$.
  We can choose  the fundamental class $[M] \in H_3(M;\IZ^{w_1(M)})$ so that it is mapped to
  $[X] \in H_3(M;\IZ^{w_1(X)})$ under the isomorphism
  $H_3^{\pi_1(M)}(\widetilde{M};\IZ^{w_1(M)}) \xrightarrow{\cong}
  H^{\pi_1(X)}_3(\widetilde{X};\IZ^{w_1(X)})$.

  Choose a vector bundle $\xi$ over $X$ with $w_1(\xi)= w_1(X)$ and
  $w_2(\xi) = w_1(X) \cup w_1(X)$.  Its pull back $f^*\xi$ satisfies
  $w_2(f^*\xi) = w_1(M) \cup w_1(M)$ and $w_1(f^*\xi) = w_1(M)$. The Wu formula, see for
  instance~\cite[Theorem~11.14 on page~132]{Milnor-Stasheff(1974)}, implies
  $w_2(TM) = w_1(TM) \cup w_1(TM)$ and hence $w_2(f^*\xi) = w_2(TM)$ and
  $w_1(f^*\xi) = w_1(TM)$.  Therefore $TM$ and $f^*\xi$ are stably isomorphic. Hence we
  can cover $f \colon M \to X$ by a bundle map
  $\overline{f} \colon TM \oplus \underline{\IR^a} \to \xi$ after possibly replacing $\xi$
  by $\xi \oplus \underline{\IR^b}$.
\end{proof}

Notice that the sphere bundle of $\xi$ is necessarily the Spivak normal bundle of
$X$. Hence we see that the Spivak normal fibration of $X$ has a vector bundle reduction.

Next we want to figure out the simple surgery obstruction
\[
  \sigma^s(f,\overline{f}) \in L_3^s(\IZ[\pi_1(X)],w_1(X))
\]
of the normal one map of degree one appearing in
Theorem~\ref{the:existence_of_a_normal_map}.  The goal is to find one $(f,\overline{f})$
such that $\sigma^s(f,\overline{f})$ vanishes. Notice that the definition of the surgery
obstruction makes sense in all dimensions, in particular also in dimension 3. For this
purpose we will need the Full  Farrell-Jones Conjecture.


\section{Short review of Farrell-Jones groups}
\label{sec:Short_review_of_Farrell-Jones_groups}

Recall that a group $G$ is called a Farrell-Jones group if it satisfies the Full
Farrell-Jones Conjecture, which means that it satisfies both the $K$-theoretic and the
$L$-theoretic Farrell-Jones Conjecture with coefficients in additive categories and with
finite wreath products. A detailed exposition on the Farrell-Jones Conjecture will be
given in~\cite{Lueck(2020book)}.

The reader does not need to know any details about the Full Farrell-Jones Conjecture since
this paper is written so that FJ can be used as a black box. We will mention the
consequences which we need in this paper when they appear.  For now, we record the following
important consequences for a torsion free Farrell-Jones group $G$.

\begin{itemize}

\item The projective class group $\widetilde{K}_0(\IZ G)$ vanishes. This implies that any
  finitely presented $n$-dimensional Poincar\'e duality group has a finite $n$-dimensional
  model for $BG$;

\item The Whitehead group $\Wh(G)$ vanishes. Hence any homotopy equivalence of finite
  $CW$-complexes with $G$ as fundamental group is a simple homotopy equivalence and every
  $h$-cobordism of dimension $\ge 6$ with $G$ as fundamental group is trivial;

\item The negative $K$-groups $K_n(\IZ G)$ for $n \le -1$ all vanish. Hence  the decorations
$L^{\epsilon}_n(\IZ G)$ in the $L$-groups do not matter;

\item The $L$-theoretic assembly map, see~\eqref{asmb},
\[
\asmb^{\epsilon}_n(G,w) \colon   H_n^G(EG;\bfL^{\epsilon}_{\IZ,w}) 
\to H_n^G(\pt;\bfL_{\IZ,w}^{\epsilon}) = L_n^{\epsilon} (\IZ G,w)
\]
is an isomorphism for $n \in \IZ$ and all decorations $\epsilon$;

\item The Borel Conjecture holds for closed aspherical manifolds of dimension $\ge 5$
  whose fundamental group is $G$.

\end{itemize}

The reader may appreciate the following status report.

\begin{theorem}[The class $\calfj$]
\label{the:status_of_the_Full_Farrell-Jones_Conjecture}
Let class $\calfj$ of Farrell-Jones groups has the following properties.

\begin{enumerate}

\item\label{the:status_of_the_Full_Farrell-Jones_Conjecture:Classes_of_groups}
The following classes of  groups belong to $\calfj$:
\begin{enumerate}

\item\label{the:status_of_the_Full_Farrell-Jones_Conjecture:Classes_of_groups:hyperbolic_groups}
Hyperbolic groups;

\item\label{the:status_of_the_Full_Farrell-Jones_Conjecture:Classes_of_groups:CAT(0)-groups}
Finite dimensional $\CAT(0)$-groups;

\item\label{the:status_of_the_Full_Farrell-Jones_Conjecture:Classes_of_groups:solvable}
Virtually solvable groups;

\item\label{the:status_of_the_Full_Farrell-Jones_Conjecture:Classes_of_groups:lattices}
  (Not necessarily cocompact) lattices in second countable locally compact Hausdorff
  groups with finitely many path components;

\item\label{the:status_of_the_Full_Farrell-Jones_Conjecture:Classes_of_groups:pi_low_dimensional}
  Fundamental groups of (not necessarily compact) connected manifolds (possibly with
  boundary) of dimension $\le 3$;

\item\label{the:status_of_the_Full_Farrell-Jones_Conjecture:Classes_of_groups:GL}
The groups $GL_n(\IQ)$ and $GL_n(F(t))$ for $F(t)$ the function field over a finite field $F$;

\item\label{the:status_of_the_Full_Farrell-Jones_Conjecture:Classes_of_groups:S-arthmetic}
$S$-arithmetic groups;

\item\label{the:status_of_the_Full_Farrell-Jones_Conjecture:Classes_of_groups:mappping_class_groups}
mapping class groups;

\end{enumerate}

\item\label{the:status_of_the_Full_Farrell-Jones_Conjecture:inheritance}
The class $\calfj$ has the following inheritance properties:

\begin{enumerate}

\item \emph{Passing to subgroups}\\
\label{the:status_of_the_Full_Farrell-Jones_Conjecture:inheritance:passing_to_subgroups}
Let $H \subseteq G$ be an inclusion of groups. 
If $G$ belongs to $\calfj$, then $H$ belongs to $\calfj$;

\item \emph{Passing to finite direct products}\\
\label{the:status_of_the_Full_Farrell-Jones_Conjecture:inheritance:Passing_to_finite_direct_products}
If the groups $G_0$ and $G_1$ belong to $\calfj$, then also $G_0 \times G_1$ belongs to $\calfj$;

\item \emph{Group extensions}\\
\label{the:status_of_the_Full_Farrell-Jones_Conjecture:inheritance:group_extensions}
Let $ 1 \to K \to G \to Q \to 1$ be an extension of groups. Suppose that for any cyclic subgroup
$C \subseteq Q$ the group $p^{-1}(C)$ belongs to $\calfj$ 
and that the group $Q$ belongs to $\calfj$.

Then $G$ belongs to $\calfj$;

\item \emph{Directed colimits}\\
\label{the:status_of_the_Full_Farrell-Jones_Conjecture:inheritance:directed_colimits}
Let $\{G_i \mid i \in I\}$ be a direct system of groups indexed by the directed set $I$
(with arbitrary structure maps). Suppose that for each $i \in I$ the group $G_i$ belongs to $\calfj$.

Then the colimit $\colim_{i \in I} G_i$ belongs to $\calfj$;

\item \emph{Passing to finite free products}\\
\label{the:status_of_the_Full_Farrell-Jones_Conjecture:inheritance:Passing_to_finite_free_products}
If  the groups $G_0$ and $G_1$ belong to $\calfj$,  then $G_0 \ast G_1 $ belongs to $\calfj$;

\item \emph{Passing to overgroups of finite index}\\
\label{the:status_of_the_Full_Farrell-Jones_Conjecture:Passing_to_over_groups_of_finite_index}
Let $G$ be an overgroup of $H$ with finite index $[G:H]$.
If $H$ belongs to $\calfj$, then $G$ belongs to $\calfj$;

\end{enumerate}

\end{enumerate}

\end{theorem}
\begin{proof}
See~\cite{Bartels-Bestvina(2016), Bartels-Echterhoff-Lueck(2008colim), Bartels-Farrell-Lueck(2014), Bartels-Lueck(2012annals),
Bartels-Lueck-Reich(2008hyper),  Bartels-Lueck-Reich-Rueping(2014),  Kammeyer-Lueck-Rueping(2016), 
Rueping(2016_S-arithmetic), Wegner(2012),Wegner(2015)}.
\end{proof}


\section{The total surgery obstruction}
\label{sec:The_total_surgery_obstruction}

The results of this section are inspired and motivated by Ranicki's total surgery
obstruction, see for instance~\cite{Kuehl-Macko-Mole(2013), Ranicki(1979),
  Ranicki(1992)}. Since we consider only aspherical Poincar\'e complexes whose fundamental
groups are Farrell-Jones groups, the exposition simplifies drastically and we get some
valuable additional information. Moreover, we get a version of Quinn's resolution obstruction
which does not require the structure of an $\ENR$ homology manifold on the relevant
Poincar\'e complexes. Then the total surgery obstruction and hence Quinn's resolution
obstruction are already determined by the symmetric signature of the finite Poincar\'e
complex.

The main result of this section will be

\begin{theorem}\label{the:Poincare_duality_groups_in_dimension_three} Let $G$ be a
  finitely presented $3$-dimensional Poincar\'e duality group which is a Farrell-Jones
  group.  

Let $X$ be a finite $3$-dimensional CW complex modeling $BG$. The following
  statements are equivalent:

\begin{enumerate}

\item\label{the:Poincare_duality_groups_in_dimension_three:one} 
There exists a closed aspherical topological manifold $N_0$ with Farrell-Jones fundamental group such that $BG \times N_0$ 
is homotopy  equivalent to a closed topological manifold;

\item\label{the:Poincare_duality_groups_in_dimension_three:all} Let $N$ be any closed
  smooth manifold, closed PL-manifold, or closed topological manifold respectively of
  dimension $\ge 2$.  Then there is exists a normal map of degree one for some vector
  bundle $\xi$ over $X$
\[
\xymatrix{TM \oplus \underline{\IR^a} \ar[r]^-{\overline{f}} \ar[d]
&
(\xi \times TN) \oplus \underline{\IR^b} \ar[d]
\\
M \ar[r]^-f
&
X \times N
}
\]
such that $M$ is a smooth manifold, PL-manifold, or topological manifold respectively and
$f$ is a simple homotopy equivalence.

\end{enumerate}
\end{theorem}


\subsection{The quadratic total surgery obstruction}
\label{subsec:The_quadratic_total_surgery_obstruction}

Let $G$ be a group together with an orientation homomorphism $w \colon G \to \{\pm 1\}$.
Then there is a covariant functor 
\[
\bfL_{\IZ,w}^{\epsilon }  \colon \Or(G) \to \SPECTRA
\]
from the orbit category to the category of spectra, where the decoration
$\epsilon$ is $\langle i \rangle$ for some $i \in \{2,1,0,-1,\ldots \} \amalg \{-\infty\}$, 
see~\cite[Definition~4.1 on page~145]{Ranicki(1992)}.
Notice that the decoration $\langle i \rangle$ for $i =2,1,0$ is also denoted by $s,h,p$
in the literature.  From $\bfL_{\IZ,w}^{\epsilon } $ we obtain a $G$-homology theory on
the category of $G$-$CW$-complexes $H^G_*(-;\bfL^{\epsilon}_{\IZ,w})$ such that for every
subgroup $H \subseteq G$ and $n \in \IZ$ we have identifications
\[
H_n^G(G/H;\bfL^{\epsilon}_{\IZ,w}) \cong \pi_n(\bfL^{\epsilon}_{\IZ,w}(G/H))
\cong L^{\epsilon}_n(\IZ H, w|_H),
\]
where $L^{\epsilon}_n(\IZ H, w|_H) $ denotes the $n$-th quadratic $L$-group with
decoration $\epsilon$ of $\IZ G$ with the $w$-twisted involution, see~\cite[Section~4
and~7] {Davis-Lueck(1998)}.  The projection $EG \to \pt$ induces the so called assembly
map
\begin{equation}
\asmb^{\epsilon}_n(G,w) \colon   H_n^G(EG;\bfL^{\epsilon}_{\IZ,w}) 
\to H_n^G(\pt;\bfL_{\IZ,w}^{\epsilon}) = L_n^{\epsilon} (\IZ G,w),
\label{asmb}
\end{equation}
which is induced by the projection $EG \to \pt$. 

In the sequel we denote for a spectrum $\bfE$ by
$\bfi(\bfE) \colon \bfE\langle 1 \rangle \to \bfE $ its $1$-connective cover. This is a
map of spectra such that $\pi_n(\bfi(\bfE))$ is an isomorphism for $n \ge 1$ and
$\pi_n(\bfE\langle 1 \rangle) = 0$ for $n \le 0$. We claim that there is a functorial
construction of the $1$-connective cover so that we get from the covariant functor
$\bfL_{\IZ,w}^{\epsilon } \colon \Or(G) \to \SPECTRA$ another covariant functor
$\bfL_{\IZ,w}^{\epsilon}\langle 1 \rangle \colon \Or(G) \to \SPECTRA$ together with a
natural transformation
$\bfi \colon \bfL_{\IZ,w}^{\epsilon}\langle 1 \rangle \to \bfL_{\IZ,w}^{\epsilon}$ such
that $\bfi(G/H)$ is a cofibration of spectra. Then we can also define a functor
$\bfL_{\IZ,w}^{\epsilon}/\bfL_{\IZ,w}^{\epsilon}\langle 1 \rangle \colon \Or(G) \to
\SPECTRA$ together with a natural transformation
$\bfpr \colon \bfL_{\IZ,w}^{\epsilon} \to
\bfL_{\IZ,w}^{\epsilon}/\bfL_{\IZ,w}^{\epsilon}\langle 1 \rangle$ such that for every
object $G/H$ in $\Or(G)$ we obtain a cofibration sequence of spectra
\[
\bfL_{\IZ,w}^{\epsilon}\langle 1 \rangle(G/H) \xrightarrow{\bfi(G/H)} \bfL_{\IZ,w}^{\epsilon}(G/H)
\xrightarrow{\bfpr(G/H)} \bfL_{\IZ,w}^{\epsilon}/\bfL_{\IZ,w}^{\epsilon}\langle 1 \rangle(G/H).
\]
For every $G$-$CW$-complex $Y$ this induces a long exact sequence
\begin{multline}
\cdots \to H_n^G(Y;\bfL^{\epsilon}_{\IZ,w}\langle 1 \rangle) \to H_n^G(Y;\bfL^{\epsilon}_{\IZ,w}) 
\\
\to H_n^G(Y;\bfL^{\epsilon}_{\IZ,w}/\bfL^{\epsilon}_{\IZ,w}\langle 1 \rangle) 
\to H_{n-1}^G(Y;\bfL^{\epsilon}_{\IZ,w}\langle 1 \rangle)
\to \cdots
\label{long_exact_sequence_for_langle_1_rangle}
\end{multline}
and we have 
\[
\pi_n(\bfL_{\IZ,w}^{\epsilon}/\bfL_{\IZ,w}^{\epsilon}\langle 1 \rangle(G/H)) \cong
\begin{cases}
L_n^{\epsilon}(H;w|_H) & n \le 0;
\\
0 & n \ge 1.
\end{cases}
\]

Now consider an aspherical oriented finite $n$-dimensional Poincar\'e complex $X$ with
universal covering $\widetilde{X} \to X$, fundamental group $G = \pi_1(X)$ and orientation
homomorphism $w = w_1(X) \colon G \to \{\pm 1\}$ in the sense
of~\cite{Wall(1967)}.  We can read $w$ from the underlying
$CW$-complex $X$ as follows.
For any abelian group $A$ we denote by $A^w$ the $\IZ G$-module whose underlying
abelian group is $A$ and on which $g \in G$ acts by multiplication with $w(g)$.
Now we use the isomorphism $H_n(C^{n-*}(\widetilde{X})) \cong_{\IZ\pi} \IZ^w$ coming from
Poincar\'e duality,
where $C^{n-*}(\widetilde{X})$ is the (untwisted) $\IZ\pi$-dual chain complex of the cellular
$\IZ\pi$-chain complex $C_*(\widetilde{X})$ of the universal covering $\widetilde{X}$.

There is an equivariant version of the Atiyah-Hirzebruch spectral sequence,
whose $E^2$-term is given by $E_{p,q}^2 =  H_{p}^G(\widetilde{X};\pi_q(\bfL^{\epsilon}_{\IZ,w}/\bfL^{\epsilon}_{\IZ,w}\langle 1 \rangle))$
and which converges to $H_{p+q}^G(\widetilde{X};\bfL^{\epsilon}_{\IZ,w}/\bfL^{\epsilon}_{\IZ,w}\langle 1 \rangle)$,
see for instance~\cite[4.7]{Davis-Lueck(1998)}. It implies 
$H_{n+1}^G(\widetilde{X};\bfL^{\epsilon}_{\IZ,w}/\bfL^{\epsilon}_{\IZ,w}\langle 1 \rangle)
= 0$ and yields an isomorphism 
\begin{equation}
H_n^G(\widetilde{X};\bfL^{\epsilon}_{\IZ,w}/\bfL^{\epsilon}_{\IZ,w}\langle 1 \rangle) 
\xrightarrow{\cong} H_n^G(\widetilde{X};L_0^{\epsilon}(\IZ)^w).
\label{lambda:1}
\end{equation}

Poincar\'e duality yields an isomorphism
\begin{equation}
H_n^G(\widetilde{X};L_0^{\epsilon}(\IZ)^w)  \xrightarrow{\cong} H^0_G(\widetilde{X};L_0^{\epsilon}(\IZ)),
\label{lambda:2}
\end{equation}
where $G$ acts trivially on $L_0^{\epsilon}(\IZ)$ in
$H^0_G(\widetilde{X};L_0^{\epsilon}(\IZ))$. There is an obvious isomorphism
\begin{equation}
H^0_G(\widetilde{X};L_0^{\epsilon}(\IZ)) \xrightarrow{\cong} H^0(X;L_0^{\epsilon}(\IZ)) \cong L_0^{\epsilon}(\IZ).
\label{lambda:3}
\end{equation}
Notice that $L_0^{\epsilon}(\IZ)$ is independent of the decoration $\epsilon$ and hence we
abbreviate $L_0(\IZ) = L_0^{\epsilon}(\IZ)$. We obtain from~\eqref{lambda:1},~\eqref{lambda:2},
and~\eqref{lambda:3} an isomorphism
\begin{equation}
H_n^G(\widetilde{X};\bfL^{\epsilon}_{\IZ,w}/\bfL^{\epsilon}_{\IZ,w}\langle 1 \rangle) \xrightarrow{\cong} L_0(\IZ).
\label{lambda:4}
\end{equation}
Its composition with $H_n^G(\widetilde{X};\bfL^{\epsilon}_{\IZ,w}) 
\to H_n^G(\widetilde{X};\bfL^{\epsilon}_{\IZ,w}/\bfL^{\epsilon}_{\IZ,w}\langle 1 \rangle)$
is denoted by
\begin{equation}
\lambda^{\epsilon}_n(X) \colon H_n^G(\widetilde{X};\bfL^{\epsilon}_{\IZ,w}) \to L_0(\IZ).
\label{lambda}
\end{equation}
From the exact sequence~\eqref{long_exact_sequence_for_langle_1_rangle} we obtain a short exact sequence
\begin{equation}
0 \to H_n^G(\widetilde{X};\bfL^{\epsilon}_{\IZ,w}\langle 1 \rangle) 
\xrightarrow{H_n^G(\id_{\widetilde{X}};\bfi)}
H_n^G(\widetilde{X};\bfL^{\epsilon}_{\IZ,w}) \xrightarrow{\lambda^{\epsilon}_n(X)}  L_0(\IZ).
\label{short_exact_sequence_to_L_0(IZ)}
\end{equation}
For every $\epsilon$ there is a natural
transformation
\[
\bfe^{\epsilon} \colon \bfL^{\epsilon} \to \bfL^{\langle -\infty \rangle}
\]
such that $e_n^{\epsilon} := \pi_n(\bfe^{\epsilon})\colon L_n^{\epsilon} (\IZ G,w)  \to L_n^{\langle - \infty \rangle}(\IZ G,w)$
is the classical change of decoration homomorphism  and the following diagram 
\[
\xymatrix@!C= 17em{%
H_n^G(\widetilde{X};\bfL^{\epsilon}_{\IZ,w}) \ar[r]^-{\asmb_n^{\epsilon}(X)} \ar[d]_{H_n^G(\id_{\widetilde{X}};\bfe^{\epsilon})}
&
H_n^G(\pt;\bfL^{\epsilon}_{\IZ,w}) = L_n^{\epsilon} (\IZ G,w) \ar[d]^{e_n^{\epsilon}}
\\
H_n^G(\widetilde{X};\bfL^{\langle -\infty\rangle}_{\IZ,w}) \ar[r]_-{\asmb^{\langle - \infty \rangle}_n(X)} 
&
H_n^G(\pt;\bfL_{\IZ,w}^{\langle - \infty \rangle}) = L_n^{\langle - \infty \rangle}(\IZ G,w) 
}
\]
commutes. Note that since $X$ is aspherical,  $G$ must be torsion free. If $G$ is a Farrell-Jones group, then
$\Wh(G)$, $\widetilde{K}_0(\IZ G)$ and $K_m(\IZ G)$ for $m \le -1$ vanish and hence all
maps in the commutative diagram are isomorphisms, in particular, the choice of the
decoration $\epsilon$ does not matter.

Let $\caln(X)$ be the set of normal bordism classes of degree one normal maps with
target $X$.  Suppose that $\caln(X)$ is not empty. Consider a normal map
$(f,\overline{f})$ of degree one with target $X$
\[
\xymatrix{TM \oplus \underline{\IR^a} \ar[r]^-{\overline{f}} \ar[d]
&
\xi \ar[d]
\\
M \ar[r]^-{f}
&
X
}
\]
One can assign to it its simple surgery obstruction
$\sigma^s(f,\overline{f}) \in L_n^s(\IZ G, w)$.  (This makes sense for all dimensions
$n$.)  Fix a normal map $(f_0,\overline{f_0})$. Then there is a commutative diagram
\begin{eqnarray}
\xymatrix@!C=16em{%
\caln^{\operatorname{TOP}}(X) \ar[r]^{\sigma^s(-,-)  - \sigma^s(f_0,\overline{f_0})}\ar[d]_{s_0}^{\cong} & L_n^s(\IZ G,w) 
\\
H_n^G(\widetilde{X};\bfL^s_{\IZ,w}\langle 1 \rangle)
\ar[r]_{H_n^G(\id_{\widetilde{X}};\bfi)}
&
H_n^G(\widetilde{X};\bfL^s_{\IZ,w}) \ar[u]_{\asmb^s_n(X)}^{\cong}
}
\label{diagram_relating_caln_and_H_n-;L)}
\end{eqnarray}
whose vertical arrows are bijections. The upper arrow sends the class of $(f,\overline{f})$ to
the difference $\sigma^s(f,\overline{f}) - \sigma^s(f,\overline{f_0})$.
This follows from the work of Ranicki~\cite[Proof of Theorem~17.4 on pages~191ff]{Ranicki(1992)}
 using~\cite[Theorem~B1]{Connolly-Davis-Khan(2014H1)}.
A detailed and careful exposition of the proof of the existence of the diagram above can be found 
in~\cite[Proposition~14.18]{Kuehl-Macko-Mole(2013)}. The right vertical arrow is an isomorphism, 
provided that $G$ is a Farrell-Jones group.

Now consider the composition
\begin{equation}
\mu^s_n(X) \colon \caln(X) \xrightarrow{\sigma^s} L_n^s(\IZ G,w) 
\xrightarrow{\asmb^s_n(X)^{-1}} H_n^G(\widetilde{X};\bfL^s_{\IZ,w})
\xrightarrow{\lambda_n^{\epsilon}(X)}  L_0(\IZ),
\label{mu_upper_s_n}
\end{equation}
where the map $\lambda_n^{\epsilon}(X)$ has been defined in~\eqref{lambda}.  From the
exact sequence~\eqref{short_exact_sequence_to_L_0(IZ)} and the
diagram~\ref{diagram_relating_caln_and_H_n-;L)} we conclude that there is precisely one
element, called the \emph{quadratic total surgery obstruction},
\begin{equation}
s(X) \in L_0(\IZ) 
\label{total_surgery_obstruction}
\end{equation}
such that for any element $[(f,\overline{f})]$ in $\caln(X)$ its image under $\mu^s_n(X)$
is $s(X)$. Moreover, we get

\begin{theorem}[The quadratic total surgery obstruction]\label{the:quadratic_total_surgery_obstruction}
  Let $X$ be an aspherical oriented finite $n$-dimensional Poincar\'e complex $X$ with
  universal covering $\widetilde{X} \to X$, fundamental group $G = \pi_1(X)$ and
  orientation homomorphism $w = w_1(X) \colon G \to \{\pm 1\}$. Suppose that $G$ is a
  Farrell-Jones group and that $\caln(X)$ is non-empty.  Then:

  \begin{enumerate}

  \item\label{the:quadratic_total_surgery_obstruction:normal}
  There exists a normal map $(f,\overline{f})$ of degree one with target $X$ 
  whose simple surgery obstruction $\sigma^s(f,\overline{f}) \in L_n^s(\IZ G,w)$ vanishes,
  if and only if $s(X) \in L_0(\IZ)$ vanishes;

  \item\label{the:quadratic_total_surgery_obstruction:vanishing} 
   If $X$ is homotopy equivalent to a closed topological manifold, then  $s(X) \in L_0(\IZ)$ vanishes.

  \end{enumerate}
\end{theorem}
\begin{proof}~\ref{the:quadratic_total_surgery_obstruction:normal}
The ``only if''-statement is obvious. The ``if''- statement is proved as follows. The vanishing of
$s(X) \in L_0(\IZ)$ implies that the element $-(f_0,\overline{f_0})$ in $\caln^{\operatorname{TOP}}(X)$
is  sent under $\mu_n^s(X)$ to zero. The exact sequence~\eqref{short_exact_sequence_to_L_0(IZ)}
implies that the composite $\caln(X) \xrightarrow{\sigma^s} L_n^s(\IZ G,w) 
\xrightarrow{\asmb^s_n(X)^{-1}} H_n^G(\widetilde{X};\bfL^s_{\IZ,w})$
sends $-(f_0,\overline{f_0})$ to an element which is in the image of
$H_n^G(\id_{\widetilde{X}};\bfi) \colon H_n^G(\widetilde{X};\bfL^s_{\IZ,w}\langle 1 \rangle)
\to H_n^G(\widetilde{X};\bfL^s_{\IZ,w})$. We conclude from the diagram~\eqref{diagram_relating_caln_and_H_n-;L)}
that $-\sigma^s(f_0,\overline{f_0})$ lies in the image of the  upper horizontal arrow  of the 
diagram~\ref{diagram_relating_caln_and_H_n-;L)}. Therefore there is an element $-(f,\overline{f})$ in $\caln^{\operatorname{TOP}}(X)$
which satisfies $\sigma^s(f,\overline{f})-\sigma^s(f_0,\overline{f_0}) = -\sigma^s(f_0,\overline{f_0})$ and hence
$\sigma^s(f,\overline{f}) = 0$.
\\[1mm]\ref{the:quadratic_total_surgery_obstruction:vanishing} 
If $X$ is simply homotopy equivalent to a closed topological manifold, then there exists an element in
$[(f,\overline{f})]$ in $\caln(X)$ with $\sigma^s(f,\overline{f}) = 0$. Now apply
assertion~\ref{the:quadratic_total_surgery_obstruction:normal}.
\end{proof}

Notice that
Theorem~\ref{the:quadratic_total_surgery_obstruction}~\eqref{the:quadratic_total_surgery_obstruction:normal}
holds also in dimensions $n \le 4$.  We are \emph{not} claiming in
Theorem~\ref{the:quadratic_total_surgery_obstruction}~\eqref{the:quadratic_total_surgery_obstruction:normal}
that that we can arrange $f$ to be a simple homotopy equivalence.  This conclusion from
the vanishing of the simple surgery obstruction does require $n \ge 5$.


\subsection{The symmetric total surgery obstruction}
\label{subsec:The_symmetric_total_surgery_obstruction}

There is also a symmetric version of the material of
Subsection~\ref{subsec:The_quadratic_total_surgery_obstruction}.  There is a covariant
functor
\[
\bfL_{\IZ,w}^{\epsilon,\sym}  \colon \Or(G) \to \SPECTRA
\]
from the orbit category to the category of spectra such that for every
subgroup $H \subseteq G$ and $n \in \IZ$ we have identifications
\[
H_n(G/H;\bfL^{\epsilon,\sym}_{\IZ,w}) \cong \pi_n(\bfL^{\epsilon,\sym}_{\IZ,w}(G/H)) 
\cong L^n_{\epsilon}(\IZ H, w|_H),
\]
where $L^n_{\epsilon}(\IZ H, w|_H)$ denotes the $4$-periodic $n$-th symmetric $L$-group with decoration
$\epsilon$ of $\IZ G$ with the $w$-twisted involution.  The projection $EG\to
\pt$ induces the symmetric assembly map
\begin{equation}
\asmb_n^{\epsilon,\sym}(X) \colon   H_n^G(EG;\bfL^{\epsilon,\sym}_{\IZ,w}) 
\to H_n^G(\pt;\bfL_{\IZ,w}^{\epsilon,\sym}) = L^n_{\epsilon} (\IZ G,w),
\label{asmb_upper-sym}
\end{equation}
which is induced by the projection $\widetilde{X} \to \pt$.

There is a natural transformation called symmetrization of covariant functors $\Or(G) \to \SPECTRA$
\begin{equation}
\bfsym^{\epsilon} \colon \bfL^{\epsilon}_{\IZ,w} \to \bfL^{\epsilon,\sym}_{\IZ,w}.
\label{bfsym}
\end{equation}
It induces the classical symmetrization homomorphisms on homotopy groups
\begin{equation}
\sym_n^{\epsilon}(G/H) \colon L_n^{\epsilon}(\IZ H,w|_H) \to L^n_{\epsilon}(\IZ H,w|_H),
\label{sym_n_upper_epsilon}
\end{equation}
which are isomorphism after inverting $2$, see~\cite[Proposition~8.2]{Ranicki(1980a)}.
We obtain a natural transformation of $G$-homology theories, see~\cite[Lemma~4.6]{Davis-Lueck(1998)}.
\begin{equation}
H_*^G(-;\bfsym^{\epsilon})\colon 
H_*^G(-;\bfL^{\epsilon}_{\IZ,w}) \to H_*^G(-;\bfL^{\epsilon,\sym}_{\IZ,w})
\label{trans_of_G-homologie_theories_sym_ast}
\end{equation}
satisfying

\begin{theorem}\label{the:H_ast_upper_G(-;bfsym_upper(epsilon)_iso_after_inverting_2}
For every $n \in \IZ$ and every $G$-$CW$-complex $X$ the maps
\[
H_*^G(-;\bfsym^{\epsilon}) \colon H_n^G(X;\bfL^{\epsilon}_{\IZ,w}) \to H_n^G(X;\bfL^{\epsilon,\sym}_{\IZ,w}) 
\]
are isomorphisms after inverting $2$.
\end{theorem}

The following diagram commutes
\begin{equation}
\xymatrix@!C=12em{%
H_n^G(\widetilde{X};\bfL^{\epsilon}_{\IZ,w}) \ar[r]^{\asmb^{\epsilon}_n(X) } \ar[d]_{H_n^G(EG,\bfsym^{\epsilon})}
& L_n^{\epsilon} (\IZ G,w) \ar[d]^{\sym_n^{\epsilon}(G/G)}
\\
H_n^G(\widetilde{X};\bfL^{\epsilon,\sym}_{\IZ,w}) \ar[r]_{\asmb_n^{\epsilon,\sym}(X)}
&
L^n_{\epsilon} (\IZ G,w).
}
\label{diagram_for_sym_and_asmb}
\end{equation}

There is an  obvious symmetric analog of 
the map~\eqref{lambda}
\begin{equation}
\lambda^{\epsilon,\sym}_n(X) \colon H_n^G(\widetilde{X};\bfL^{\epsilon,\sym}_{\IZ,w}) \to L^0(\IZ),
\label{lambda_sym}
\end{equation}
and  of the short exact 
sequence~\eqref{short_exact_sequence_to_L_0(IZ)}
\begin{equation}
0 \to H_n^G(\widetilde{X};\bfL^{\epsilon,\sym}_{\IZ,w}\langle 1 \rangle) 
\xrightarrow{H_n^G(\id_{\widetilde{X}};\bfi)}
H_n^G(\widetilde{X};\bfL^{\epsilon,\sym}_{\IZ,w}) \xrightarrow{\lambda^{\epsilon,\sym}_n(X)}  L^0(\IZ).
\label{short_exact_sequence_to_L_upper_0(IZ)}
\end{equation}
The following diagram 
\begin{equation}
\xymatrix@!C= 11em{%
0 \longrightarrow 
H_n^G(\widetilde{X};\bfL^{\epsilon}_{\IZ,w}\langle 1 \rangle) 
\ar[r]^-{H_n^G(\id_{\widetilde{X}};\bfi)}
\ar[d]^{H_n^G(\id;\bfsym^{\epsilon}\langle 1 \rangle)}
&
H_n^G(\widetilde{X};\bfL^{\epsilon}_{\IZ,w}) \ar[r]^-{\lambda_n^{\epsilon}(X)}  
\ar[d]^{H_n^G(\id_{\widetilde{X}};\bfsym^{\epsilon})}
&
L_0(\IZ) \ar[d]^{\sym_0}
\\
0 \longrightarrow 
H_n^G(\widetilde{X};\bfL^{\epsilon,\sym}_{\IZ,w}\langle 1 \rangle) 
\ar[r]^-{H_n^G(\id_{\widetilde{X}};\bfi)}
&
H_n^G(\widetilde{X};\bfL^{\epsilon,\sym}_{\IZ,w}) \ar[r]^-{\lambda^{\epsilon,\sym}_n(X)}  
&
L^0(\IZ) 
}
\label{diagram_relating_quadratic_and-symmetric}
\end{equation}
commutes, has exact rows, and all its vertical arrows are bijections after inverting $2$
since the map~\eqref{sym_n_upper_epsilon} is bijective after 
inverting $2$ and we have~\cite[Theorem~4.7] {Davis-Lueck(1998)}.
Under the standard identifications 
\begin{eqnarray}
h_0 \colon L_0(\IZ) & \xrightarrow{\cong}& \IZ;
\label{h_0}
\\
h^0 \colon L^0(\IZ) & \xrightarrow{\cong} & \IZ,
\label{h_upper_0}
\end{eqnarray}
the map $\sym_0 \colon L_0(\IZ) \to L^0(\IZ)$ becomes $8 \cdot \id \colon \IZ \to \IZ$, see the proof 
of~\cite[Proposition~8.2]{Ranicki(1980a)}, and hence is injective.  Define the \emph{symmetric total surgery obstruction}
\begin{equation}
s^{\sym}(X) \in L^0(\IZ) 
\label{symmetric_total_surgery_obstruction}
\end{equation}
to be the image of $s(X)$ defined in~\eqref{total_surgery_obstruction} under the injection
$\sym_0 \colon L_0(\IZ) \to L^0(\IZ)$.
Theorem~\ref{the:quadratic_total_surgery_obstruction} implies

\begin{theorem}[The symmetric total surgery obstruction]\label{the:symmetric_total_surgery_obstruction}
  Let $X$ be an aspherical oriented finite $n$-dimensional Poincar\'e complex $X$ with
  universal covering $\widetilde{X} \to X$, fundamental group $G = \pi_1(X)$ and
  orientation homomorphisms $w = w_1(X) \colon G \to \{\pm 1\}$. Suppose that $G$ is a
  Farrell-Jones group and that $\caln(X)$ is non-empty. Then 

  \begin{enumerate}

  \item\label{the:symmetric_total_surgery_obstruction:normal}
  There exists a normal map of degree one $(f,\overline{f})$ with target $X$ 
  whose simple surgery obstruction $\sigma^s(f,\overline{f}) \in L_n^s(\IZ G,w)$ vanishes,
  if and only if $s^{\sym}(X) \in L^0(\IZ)$ vanishes;

  \item\label{the:symmetric_total_surgery_obstruction:vanishing} 
   If  $X$ is homotopy equivalent to a closed topological manifold, then  $s^{\sym}(X) \in L^0(\IZ)$ vanishes.

  \end{enumerate}
  \end{theorem}

Now we study the main properties of the symmetric total surgery obstruction.

If $A$ is an abelian group, denote by $A\twotor$ its quotient by the abelian subgroup of
elements in $A$, whose order is finite and a power of two. For an element $a \in A$ denote
by $[a]_2$ its image under the projection $A \to A\twotor$.
 
Next we show that $s^{\sym}(X)$ and $s(X)$ are determined by the image
$[\sigma^{s,\sym}_G(\widetilde{X})]_2$ of 
$\sigma^{s,\sym}_G(\widetilde{X})$ under $L^n_s (\IZ G,w) \to L^n_s (\IZ G,w)\twotor$,
where $\sigma^{s,\sym}_G(\widetilde{X})$ is the symmetric signature in the sense 
of~\cite[Proposition~6.3]{Ranicki(1980a)} taking into account, 
that $G$ is a torsionfree Farrell-Jones group and hence the decorations do not matter.

\begin{theorem}\label{the:symmetric_signature_determines_symmetric_total_surgery_obstruction}
  Let $X$ be an aspherical oriented finite $n$-dimensional Poincar\'e complex $X$ with
  universal covering $\widetilde{X} \to X$, fundamental group $G = \pi_1(X)$ and
  orientation homomorphism $w = w_1(X) \colon G \to \{\pm 1\}$. Suppose that $G$ is a
  Farrell-Jones group and that $\caln(X)$ is non-empty.

Then there is precisely one element $u \in H_n^G(\widetilde{X};\bfL^{s,\sym}_{\IZ,w})\twotor$
such that the injective map
\[
\asmb_n^{s,\sym}(X)\twotor \colon H_n^G(\widetilde{X};\bfL^{s,\sym}_{\IZ,w})\twotor 
\to  L^n_{s} (\IZ G,w)\twotor
\]
sends $u$ to the element $[\sigma_G^{s,\sym}(\widetilde{X})]_2$ 
associated to  the symmetric signature $\sigma_G^{s,\sym}(\widetilde{X})$,
and the composite
\[
H_n^G(\widetilde{X};\bfL^{s,\sym}_{\IZ,w})\twotor 
\xrightarrow{\lambda^{s,\sym}(\widetilde{X})\twotor}L^0(\IZ)\twotor 
\xrightarrow{h_0\twotor} \IZ\twotor = \IZ
\]
sends $u$ to  $1-h^0(s^{\sym}(X)) = 1 - 8 \cdot h_0(s(X))$.
\end{theorem}
\begin{proof}
  Consider a normal map$(f,\overline{f})$ of degree one from $M$ to $X$. 

  Since $G$ is a Farrell-Jones group, the assembly map $\asmb$ of~\eqref{asmb} is
  bijective for all $n \in \IZ$.  The homomorphism
 $\sym^s_n \colon L_n^s(\IZ G,w) \to L^n_s(\IZ G,w)$ sends
$\sigma^s(f,\overline{f})$ to
$\sigma_G^{s,\sym}(\overline{M}) - \sigma_G^{s,\sym}(\widetilde{X})$, where
$\overline{M} \to M$ is the pull back of the $G$-covering $\widetilde{X} \to BG$ by $f$,
see~\cite[Section~6]{Ranicki(1980b)}.  
We conclude from the commutative diagram~\eqref{diagram_for_sym_and_asmb} 
that there is an element
$u' \in H_n^G(\widetilde{X};\bfL^{s,\sym}_{\IZ,w})\twotor$ whose image 
under $\asmb_n^{s,\sym}(\widetilde{X})\twotor$ is
$[\sigma_G^{s,\sym}(\overline{M})]_2 - [\sigma_G^{s,\sym}(\widetilde{X})]_2$.

We conclude from the commutative 
diagram~\eqref{diagram_for_sym_and_asmb} 
that the assembly map
\[
\asmb_n^{\epsilon,\sym}(\widetilde{X}) \colon   H_n^G(\widetilde{X};\bfL^{\epsilon,\sym}_{\IZ,w})
\to H_n^G(\pt;\bfL^{\epsilon,\sym}) = L^n_{\epsilon} (\IZ G,w)
\]
of~\eqref{asmb_upper-sym} is an isomorphism after inverting $2$
since the  upper horizontal arrow is the bijective map~\eqref{asmb},
 and the two vertical arrows are isomorphisms after inverting $2$,
see~\eqref{sym_n_upper_epsilon} and  Theorem~\ref{the:H_ast_upper_G(-;bfsym_upper(epsilon)_iso_after_inverting_2}.
Hence the map 
\[
\asmb_n^{s,\sym}(\widetilde{X})\twotor \colon H_n^G(\widetilde{X};\bfL^{s,\sym}_{\IZ,w})\twotor 
\to  L^n_{s} (\IZ G,w)\twotor
\]
is injective.

We conclude from the diagram~\eqref{diagram_relating_quadratic_and-symmetric} that the image of $u'$
under the composite
\[
H_n^G(\widetilde{X};\bfL^{s,\sym}_{\IZ,w})\twotor \xrightarrow{\lambda^{s,\sym}(\widetilde{X})\twotor}L^0(\IZ)\twotor 
\xrightarrow{h_0\twotor} \IZ\twotor = \IZ
\]
is $h^0(s^{\sym}(X))$. We have
$8 \cdot h_0(s(X)) = h^0(s^{\sym}(X))$. Hence it suffices to show
that there is an element $u'' \in H_n^G(\widetilde{X};\bfL^{s,\sym}_{\IZ,w})\twotor$ such
that its image under the map $\asmb_n^{s,\sym}(\widetilde{X})\twotor$ is
$[\sigma_G^{s,\sym}(\overline{M})]_2$ and the image of $u''$ under the composite
\[
H_n^G(\widetilde{X};\bfL^{s,\sym}_{\IZ,w})\twotor \xrightarrow{\lambda_n^{s,\sym}(X)\twotor}L^0(\IZ)\twotor 
\xrightarrow{h^0\twotor} \IZ\twotor = \IZ
\]
is $1$ since then  we can take $u = u'' -u'$.

For simplicity we give the proof of the existence of the element $u''$ only in the special
case, where $w$ is trivial.  For every $n \ge 0$ and every
connected $CW$-complex $X$, the symmetric signature defines a map, see~\cite[Proposition~6.3]{Ranicki(1980b)},
\[
\sigma_n^{s,\sym}(X)\colon \Omega_n^{\operatorname{TOP}}(X) \to L^n_s(\IZ[\pi_1(X)]) ,\quad 
[f \colon M \to X] \mapsto \sigma_G^{s,\sym}(\overline{M}).
\]
Without giving the details of the proof, we claim that this natural transformation of
functors from the category of $CW$-complexes to the category of $\IZ$-graded abelian
groups can be implemented as a functor from the category of $CW$-complexes to the category
of spectra. We conclude from the general theory about assembly maps,
see~\cite[Section~6]{Davis-Lueck(1998)} or~\cite{Weiss-Williams(1994a)}, that we can lift
$\sigma_n^{s,\sym}(X)$ over $\asmb_n^{s,\sym}(X)$ to a map $\tau_n^{s,\sym}(X)$
\[
\xymatrix@!C=12em{
& H_n^G(\widetilde{X};\bfL^{s,\sym}_{\IZ,w}) \ar[d]^{\asmb_n^{s,\sym}(X)}
\\
\Omega^{\operatorname{TOP}}_n(X) \ar[r]_{\sigma^{s,\sym}_n(X)} 
\ar[ru]^-{\tau_n^{s,\sym}(X)}
&
 L^n_{h} (\IZ G)
}
\]
such that $\tau_*^{s,\sym}(-)$ is a transformation of homology theories. Such a construction seems also to be contained 
in~\cite{Laures-McClure(2014)}. Consider the map
\[
\nu_n(X) \colon \Omega^{\operatorname{TOP}}_n(X) \xrightarrow{d_n} H_n(X;\IZ) 
\xrightarrow{- \cap  [X]} H^0(X;\IZ) \to \IZ,
\]
where the first map $d_n$ sends $[f \colon M \to X]$ to $f_*([M])$. The naturality of the
Atiyah-Hirzebruch spectral sequence implies that the following diagram
\[\xymatrix@!C=10em{%
\Omega_n^{\operatorname{TOP}}(X) \ar[r]^-{\nu_n(X)}  \ar[d]_{\tau_n^{s,\sym}(X)}
& 
\IZ 
\\
H_n^G(\widetilde{X};\bfL^{s,\sym}_{\IZ,w}) \ar[r]_-{\lambda^{s,\sym}(X)} 
&
L^0(\pt) 
\ar[u]_{h^0}
}
\]
 commutes. Define $u''$ to be the image of $f \colon M \to X$ under the composite
\[
\Omega_n^{\operatorname{TOP}}(X) \xrightarrow{\tau_n^{s,\sym}(X)} H_n^G(\widetilde{X};\bfL^{s,\sym}_{\IZ,w})  
\to H_n^G(\widetilde{X};\bfL^{s,\sym}_{\IZ,w})\twotor.
\]
Since the degree of $f \colon M \to X$ is one, the image of $[f\colon M \to X]$ under $\nu_n(X)$ is $1$.
An easy diagram chase shows that $u''$ has the desired properties.
This finishes the proof of Theorem~\ref{the:symmetric_signature_determines_symmetric_total_surgery_obstruction}.
\end{proof}

Theorem~\ref{the:symmetric_signature_determines_symmetric_total_surgery_obstruction}
together with the homotopy invariance of the symmetric signature implies the homotopy
invariance of the total surgery obstruction. More precisely, we have

\begin{theorem}[Homotopy invariance of the total surgery obstruction]
\label{the:s(X)_is_homotopy_invariant}
Let $X$ be an aspherical oriented finite $n$-dimensional Poincar\'e complex 
such that $\pi_1(X)$ is a Farrell-Jones group and  $\caln(X)$ is non-empty. 
Let $Y$ be a finite $n$-dimensional $CW$-complex  which is homotopy equivalent to $X$.

Then $Y$ is an aspherical oriented finite $n$-dimensional Poincar\'e complex
such that $\pi_1(Y)$ is a   Farrell-Jones group and such that $\caln(Y)$ is non-empty. We get
\begin{eqnarray*}
s(X) & = & s(Y);
\\
s^{\sym}(X) & = & s^{\sym}(Y).
\end{eqnarray*}
\end{theorem}
\begin{proof} Choose a homotopy equivalence $f \colon X \to Y$. 
Define $w_1(Y) \in H^n(Y;\IZ/2)$ to be $f^*w_1(X)$. 
Then obviously $Y$ inherits the structure of an oriented finite $n$-dimensional Poincar\'e complex
from $X$ if we take as fundamental class $[Y]$ the image of $[X]$ under the isomorphism
$H_n^{\pi_1(X)}(\widetilde{X};\IZ^{w_1(X)}) \xrightarrow{\cong} H_n^{\pi_1(Y}(\widetilde{Y};\IZ^{w_1(Y)})$
 induced by $f$. A consequence of the  basic features of the symmetric signature and $G$ being a torsionfree Farrell-Jones group 
is that the isomorphism induced by $\pi_1(f)$ 
\[
L^n_s(\IZ[\pi_1(X)],w_1(X)) \xrightarrow{\cong} L^n_s(\IZ[\pi_1(Y)],w_1(Y))
\]
sends $\sigma^{s,\sym}_{\pi_1(X)}(\widetilde{X})$ to  $\sigma^{s,\sym}_{\pi_1(Y)}(\widetilde{Y})$.
Now apply  Theorem~\ref{the:symmetric_signature_determines_symmetric_total_surgery_obstruction}.
\end{proof}

Next we show a product formula.

\begin{theorem}[Product formula]
\label{the:product_formula}
For each $i \in \{0,1\}$, let $X_i$ be an aspherical oriented finite $n_i$-dimensional Poincar\'e complex  with
fundamental group $G_i = \pi_1(X_i)$ and
orientation homomorphism $v_i :=w_1(X_i) \colon G_i \to \{\pm 1\}$ such that $G_i$ is a
Farrell-Jones group and that $\caln(X_i)$ is non-empty. 

Then $X_0 \times X_1$ is an aspherical oriented finite $(n_0+n_1)$-dimensional Poincar\'e complex with
fundamental group $G_0 \times G_1$ and
orientation homomorphisms $v := w_1(X \times N) \colon G_0 \times G_1 \to \{\pm 1\}$ sending 
$(g_0,g_1)$ to $v_0(g_0)  \cdot v_1(g_1)$
such that $G_0 \times G_1$ is a
Farrell-Jones group and that $\caln(X_0 \times X_1)$ is non-empty, and we get in $\IZ$ 
\[
(1 -8 \cdot h^0(s(X_0 \times X_1)))  =  (1 - 8 \cdot  h^0(s(X_0)))  \cdot (1 - 8 \cdot h^0(s(X_1))).
\]
\end{theorem}
\begin{proof} The product $G_0 \times G_1$ is a Farrell-Jones group by
Theorem~\ref{the:status_of_the_Full_Farrell-Jones_Conjecture}~%
\eqref{the:status_of_the_Full_Farrell-Jones_Conjecture:inheritance:Passing_to_finite_direct_products}.

The tensor product gives a pairing, see~\cite[Section~8]{Ranicki(1980a)},
\begin{equation}
\otimes \colon L^{n_0}_s(\IZ G_0,v_0) \otimes L^{n_1}_s(\IZ G_1,v_1) \to L^{n_0+n_1}(\IZ[G_0,\times G_1],v).
\label{otimes_on_L_sym}
\end{equation}
Now we claim that there is a pairing
\[
\times \colon H_{n_0}^{G_0}(\widetilde{X_0};\bfL_{\IZ,v_0}^{s,\sym}) \otimes
H_{n_1}^{G_1}(\widetilde{X_1};\bfL_{\IZ,v_1}^{s,\sym}) \to
H_{n_0+ n_1}^{G_0 \times G_1}(\widetilde{X_0 \times X_1};\bfL_{\IZ,v}^{s,\sym})
\]
such that the following diagram commutes
\begin{equation}
\xymatrix@!C=16em{L^{n_0}_s(\IZ G_0,v_0) \otimes L^{n_1}_s(\IZ G_1,v_1) 
\ar[r]^-{\otimes}
&
 L^{n_0+n_1}(\IZ[G_0,\times G_1],v)
\\
H_{n_0}^{G_0}(\widetilde{X_0};\bfL_{\IZ,v_0}^{s,\sym}) \otimes
H_{n_1}^{G_1}(\widetilde{X_1};\bfL_{\IZ,v_1}^{s,\sym}) 
\ar[r]^-{\times} \ar[u]^{\asmb_{n_0}^{\epsilon,\sym}(\widetilde{X_0}) \otimes \asmb_{n_1}^{\epsilon,\sym}(\widetilde{X_1})}
\ar[d]_{\lambda_{n_0}^{s,\sym}(X_0) \otimes \lambda_{n_1}^{s,\sym}(X_1)} 
&
H_{n_0+n_1}^{G_0 \times G_1}(\widetilde{X_0 \times X_1};\bfL_{\IZ,v}^{s,\sym})
\ar[u]_{\asmb_{n_0+n_1}^{\epsilon,\sym}(\widetilde{X_0 \times X_1})}  \ar[d]^{\lambda^{s,\sym}_{n_0 +n_1}(X_0 \times X_1)}
\\
L^0(\IZ) \otimes L^0(\IZ) \ar[r]^-{\otimes} \ar[d]_{h^0 \otimes h^0}^{\cong}
&
L^0(\IZ) \ar[d]^{h^0}_{\cong}
\\
\IZ \otimes \IZ \ar[r]
&
\IZ
}
\label{commutative_diagram_products}
\end{equation}
where the lowermost horizontal arrow is the multiplication on $\IZ$. In order to get this diagram,
one has firstly  to promote  the functor 
\[
\bfL_{\IZ,w}^{\epsilon,\sym}  \colon \Or(G) \to \SPECTRA
\]
to a functor 
\[
\bfL_{\IZ,w}^{\epsilon,\sym}  \colon \Or(G) \to \SPECTRA^{\sym}
\]
to the category $\SPECTRA^{\sym}$ of symmetric spectra. Notice that the advantage of
$\SPECTRA^{\sym}$ in comparison with $\SPECTRA$ is that $\SPECTRA^{\sym}$ has a functorial
smash product $\wedge$. In the second step one has to construct a map of spectra
\[
\bfL_{\IZ,w}^{\epsilon,\sym}(G/H_0)  \wedge \bfL_{\IZ,w}^{\epsilon,\sym}(G/H_1) 
\to \bfL_{\IZ,w}^{\epsilon,\sym}((G \times G)/(H_0 \times H_1)),
\]
which on homotopy groups induces the map
\[
\otimes \colon L^{n_0}_s(\IZ H_0,v_0|_{H_0}) \otimes L^{n_1}_s(\IZ H_1,v_1|_{H_1}) 
\to 
L_s^{n_0+n_1}(\IZ[H_0\times H_1],v|_{H_0 \times H_1})
\]
under the identifications
\begin{eqnarray*}
\pi_k(\bfL_{\IZ,w}^{\epsilon,\sym}(G/H_0)) 
& \cong & 
L^{k}_s(\IZ H_0,v_0|_{H_0});
\\
\pi_k(\bfL_{\IZ,w}^{\epsilon,\sym}(G/H_1)) 
& \cong & 
L^{k}_s(\IZ H_1,v_1|_{H_1});
\\
\pi_k(\bfL_{\IZ,w}^{\epsilon,\sym}(G \times G/H_0 \times H_1)) 
& \cong & 
L^{k}_s(\IZ[H_0 \times H_1],v|_{H_0 \times H_1}),
\end{eqnarray*}
and are natural in $G/H_0$ and $G/H_1$. We omit the details of this construction, see also 
Remark~\ref{the:Special_case_of_Theorem_ref(the:product_formula)}.
Now the claim follows from Theorem~\ref{the:symmetric_signature_determines_symmetric_total_surgery_obstruction}
and the product formula for the symmetric signature, see~\cite[Proposition~8.1~(i)]{Ranicki(1980b)}, 
which says that the pairing~\eqref{otimes_on_L_sym}
sends $\sigma_{G_0}^{s,\sym}(\widetilde{X_0}) \otimes \sigma_{G_1}^{s,\sym}(\widetilde{X_1})$ to
$\sigma_{G_0 \times G_1}^{s,\sym}(\widetilde{X_0 \times X_1})$.

\cite{Laures-McClure(2014)}
\end{proof}

\begin{remark}[Special case of Theorem~\ref{the:product_formula}]%
\label{the:Special_case_of_Theorem_ref(the:product_formula)}
  In the proof of Theorem~\ref{the:product_formula} we have not given
  the details of the proof of the existence of the commutative
  diagram~\eqref{commutative_diagram_products}. We will need Theorem~\ref{the:product_formula} only in the
  special case, where $n_0 = 3$ and $X_1$ is a closed $n$-dimensional manifold
  and then the desired assertion is 
  \[
   s^{\sym}(X_0) = s^{\sym}(X_0 \times X_1).
  \]
  For the reader's convenience we give a direct complete proof in this special case.  We
  have $L^0(\IZ) \cong \IZ$, $L^1(\IZ) \cong \IZ/2$ and $L^i(\IZ) = 0$ for $i = 1,2,$
 see~\cite[Proposition~7.2]{Ranicki(1980a)}.
  The  Atiyah-Hirzebruch spectral sequence shows that the map $\lambda^{\epsilon,\sym}_n(X_0)$
  of~\eqref{lambda_sym} induces an isomorphism
  \[
  \lambda^{\epsilon,\sym}_{n_0}(X_0)\twotor \colon H_{n_0}^{G_0}(\widetilde{X_0};\bfL^{\epsilon,\sym}_{\IZ,w})\twotor 
\to L^0(\IZ)\twotor = L^0(\IZ)
  \]
  since we assume $n_0 = 3$.  We have already shown in
  Theorem~\ref{the:symmetric_signature_determines_symmetric_total_surgery_obstruction} that
  \[
\asmb_{n_0}^{s,\sym}(X_0)\twotor \colon H_{n_0}^{G_0}(\widetilde{X_0};\bfL^{s,\sym}_{\IZ,w})\twotor 
\to  L^{n_0}_{s} (\IZ G_0,v_0)\twotor
\]
is injective and that there is a unique element
$u_0 \in H_n^{G_0}(\widetilde{X_0};\bfL^{s,\sym}_{\IZ,v_0})$ which is mapped to
$1-h^0(s^{\sym}(X_0))$ and to $[\sigma^{s,\sym}_{G_0}(\widetilde{X_0})]_2$ under these maps.
Let $(f_0,\overline{f_0})$ be a normal map from a closed $3$-manifold $M_0$ to $X_0$.  We
have explained in the proof of
Theorem~\ref{the:symmetric_signature_determines_symmetric_total_surgery_obstruction} that
there is an element $u_0'' \in H_{n_0}^G(\widetilde{X_0};\bfL^{s,\sym}_{\IZ,w})$ whose image
under $\asmb_{n_0}^{s,\sym}(X_0)\twotor \colon H_{n_0}^{G_0}(\widetilde{X_0};\bfL^{s,\sym}_{\IZ,w})\twotor
\to  L^{n_0}_{s} (\IZ G_0,v_0)\twotor$ is $\sigma^{s,\sym}_{G_0}(\overline{M_0})$ for the $G_0$-covering $\overline{M_0} \to M_0$
given by the pullback of $\widetilde{X_0} \to X_0$ with $f_0$ and whose image under the
isomorphism
\[
h^0 \circ \lambda^{\epsilon,\sym}_{n_0}(X_0)\twotor \colon H_n^{G_0}(\widetilde{X_0};\bfL^{\epsilon,\sym}_{\IZ,v_0})\twotor 
\xrightarrow{\cong} \IZ
\]
 is $1$. Hence we get
\[
[\sigma_{G_0}^{s,\sym}(\widetilde{X_0})]_2 = (1 - h^0(s^{\sym}(X_0))) \cdot [\sigma_{G_0}^{s,\sym}(\overline{M_0})]_2.
\]
We conclude from the product formula for the symmetric signature, see~\cite[Proposition~8.1~(i)]{Ranicki(1980b)}, 
\begin{eqnarray}
\label{lidhiudrgudhkguh}
\quad \sigma^{s,\sym}_{G_0 \times G_1}(\widetilde{X_0 \times X_1})
& = & 
\sigma^{s,\sym}_{G_0}(\widetilde{X_0}) \otimes \sigma^{s,\sym}_{G_1}(\widetilde{X_1})
\\
& = &
(1 - h^0(s^{\sym}(X_0)) \cdot \sigma^{s,\sym}_{G_0}(\overline{M_0})) \otimes \sigma^{s,\sym}_{G_1}(X_1)
\nonumber \\
& = &
(1-h^0(s^{\sym}(X_0))) \cdot \sigma^{s,\sym}_{G_0 \times G_1}(\overline{M_0} \times \widetilde{X_1}).
\nonumber
\end{eqnarray}
As we have explained in the proof of
Theorem~\ref{the:symmetric_signature_determines_symmetric_total_surgery_obstruction},
there exists a unique element
$u'' \in H_n^{G_0 \times G_1}(\widetilde{X_0} \times \widetilde{X_1}
;\bfL^{s,\sym}_{\IZ,w})\twotor$, whose image under
\begin{multline*}
\asmb^{s,\sym}_{n_0 +n_1}(X_0 \times X_1)\twotor 
\colon H_{n_0+n_1}^{G_0 \times G_1}(\widetilde{X_0 \times X_1};\bfL^{\epsilon,\sym}_{\IZ,w})\twotor 
\\
\to L^{n_0+n_1}_{s} (\IZ[G_0 \times G_1],v_0 \times v_1)\twotor
\end{multline*}
is 
$[\sigma^{s,\sym}_{G_0 \times G_1}(\overline{M_0} \times \widetilde{X_1})]_2$ and whose image under 
\[h^0 \circ \lambda^{\epsilon,\sym}_n(X_0 \times X_1)\twotor \colon 
H_n^{G_0 \times G_1}(\widetilde{X_0 \times X_1};\bfL^{\epsilon,\sym}_{\IZ,w})\twotor 
\to \IZ
\]
is $1$. Here we use that $X_1$ and hence $M_0 \times X_1$ is a closed manifold.
Theorem~\ref{the:symmetric_signature_determines_symmetric_total_surgery_obstruction} 
together with~\eqref{lidhiudrgudhkguh} implies
\[
 (1-h^0(s^{\sym}(X_0 \times X_1))) = 1-h^0(s^{\sym}(X_0))).
\]
Hence we get $s^{\sym}(X_0) = s^{\sym}(X_0 \times X_1)$.
\end{remark}


\subsection{Proof of Theorem~\ref{the:Poincare_duality_groups_in_dimension_three}}

\begin{proof}[Proof of Theorem~\ref{the:Poincare_duality_groups_in_dimension_three}]
  Recall from Subsection~\ref{subsec:Basic_facts_about_Poincare_duality_groups}
  that there is a finite $3$-dimensional Poincar\'e complex model $X$ for $BG$.
  Also recall  from Theorem~\ref{the:existence_of_a_normal_map} that $\caln(BG)$ is non-empty.
  The implication $\ref{the:Poincare_duality_groups_in_dimension_three:all} 
  \implies\ref{the:Poincare_duality_groups_in_dimension_three:one}$ is obviously true;
  the implication $\ref{the:Poincare_duality_groups_in_dimension_three:one} 
  \implies\ref{the:Poincare_duality_groups_in_dimension_three:all}$
  is proved as follows.
 
  By assumption there are closed aspherical topological manifolds $M_0$ and $N_0$ and a
  homotopy equivalence $f_0 \colon M_0 \to X \times N_0$. We conclude from
  Theorem~\ref{the:s(X)_is_homotopy_invariant} that
  $s^{\sym}(M_0) = s^{\sym}(X \times N_0)$. Since $M_0$ and $N_0$ are 
  closed aspherical topological manifolds, $s^{\sym}(M_0)$ and $s^{\sym}(N_0)$ 
  vanish by  Theorem~\ref{the:symmetric_total_surgery_obstruction}~%
\eqref{the:symmetric_total_surgery_obstruction:vanishing}.
   We conclude from Theorem~\ref{the:product_formula}, or
  just from Remark~\ref{the:Special_case_of_Theorem_ref(the:product_formula)}, that
  $s^{\sym}(X) = 0$.  From Theorem~\ref{the:symmetric_total_surgery_obstruction}~%
\eqref{the:symmetric_total_surgery_obstruction:normal}  we
  obtain a normal map of degree one $(f,\overline{f})$ with target $X$ and vanishing
  simple surgery obstruction $\sigma^s(f,\overline{f}) \in L_3^s(\IZ G,w_1(X))$.  Let $N$ be a
  closed smooth manifold, closed PL-manifold, or closed topological manifold respectively
  of dimension $\ge 2$.  By the product formula for the surgery obstruction,
  see~\cite[Proposition~8.1(ii)]{Ranicki(1980b)}, the surgery obstruction of the normal map of
  degree one $(f \times \id_{N}, \overline{f} \times \id_{TN})$ obtained by crossing
  $(f,\overline{f})$ with $N$ is trivial. Since the dimension of $X \times N$ is greater
  or equal to $5$, we can do surgery in the smooth, PL, or topological category
  respectively to arrange that $f \times \id_{N}$ is a simple homotopy equivalence with a
  closed smooth manifold, closed PL-manifold, or closed topological manifold respectively
  as source.
\end{proof}


\section{Short review of $\ENR$ homology manifolds}
\label{sec:Short_review_of_ENR-homology_manifolds}

A topological space $X$ is called a
\emph{Euclidean neighborhood retract} 
or briefly an $\ENR$ if it $X$ is homeomorphic to a closed subset $X'$ of some Euclidean space $\mathbb R^n$
such that $X'$ has an open neighborhood $U$ in $\mathbb R^n$ that retracts to $X$. Such a space is finite-dimensional, metrizable, separable, locally compact, and locally contractible. It is an illuminating exercise using the Tietze Extension Theorem to show that if such an $X$ is embedded as a closed subset of any normal space, then $X$ is a neighborhood retract in that space. 

A theorem of Borsuk says that every finite-dimensional, metrizable, separable, locally compact, and locally contractible space $X$ is an $\ENR$. The one-point compactification of such an $X$ is finite-dimensional and therefore embeds in a finite-dimensional sphere. Removing the point at infinity from both the one-point compactification and the sphere yields a closed embedding of $X$ into a Euclidean space. In Theorem A7 of \cite{Hatcher(2005)} a neighborhood retraction is constructed for compact $X$. The argument given extends easily to noncompact $X$. Hatcher focuses on the compact case in order to emphasize that compact $\ENR$s have finitely presented fundamental groups and finitely generated homology groups.


\begin{definition} [$\ENR$ homology manifold]\label{def:homology-ENR-manifold}
  A  \emph{$n$-dimensional $\ENR$ homology mani\-fold $X$ (without boundary)} is an $\ENR$  such that 
 for every $x \in X$ the $i$-th singular homology group 
        $H_i(X,X-\{x\})$ is trivial
        for $i \not= n$ and infinite cyclic for $i = n$. We call $X$ \emph{closed} if it is compact.
\end{definition}

An $\ENR$ homology manifold in the sense of 
Definition~\ref{def:homology-ENR-manifold}
is the same as a generalized manifold in the sense of
Daverman~\cite[page~191]{Daverman(1986)}, as pointed out 
in~\cite[page~3]{Bartels-Lueck-Weinberger(2010)}.
Every closed $n$-dimensional topological manifold is 
a closed $n$-dimensional $\ENR$ homology manifold
(see~\cite[Corollary 1A in V.26 page~191]{Daverman(1986)}). 

\begin{definition}[DDP]\label{def:DDP} An $\ENR$ homology manifold $M$ is said to have the 
\emph{disjoint disk property (DDP)}, if for one (and hence any) choice of metric on $M$,
any $\epsilon > 0$
and any  maps $f, g \colon D^2 \to M$, there are maps
$f', g' \colon D^2 \to M$ so that $f'$ is $\epsilon$-close to $f$,
$g'$ is $\epsilon$-close to $g$ and $f'(D^2) \cap g'(D^2) = \emptyset$.
\end{definition}

\begin{definition} [$\ENR$ homology manifold with $\ENR$ boundary]\label{def:homology-ENR-manifold_with_boundary}
   An \emph{$n$-dimensional $\ENR$ homology manifold $X$ with boundary
  $\partial X$} is an $\ENR$ $X$ which is a
  disjoint union $X = \inti X \cup \partial X$, where
  \begin{itemize}
  \item $\inti X$ is an $n$-dimensional $\ENR$ homology manifold, the "interior" of the homology manifold with boundary $X$;
  \item $\partial X$ is an $(n-1)$-dimensional $\ENR$ homology manifold;
  \item for every $z \in \partial X$ the singular homology group
        $H_i(X, X \setminus \{ z \})$ vanishes for all $i$. 
  \end{itemize}
\end{definition}

This definition is rather general. It includes the ``bad'' closed complementary domain of an Alexander Horned sphere embedded in $S^3$. In our main application, however, the boundary will be a Z-set (see below) in $X$.


\section{A stable $\ENR$-version of the Cannon Conjecture}
\label{sec:stable_version_of_the_Cannon_Conjecture}

\begin{theorem}[Stable $\ENR$-version of the Cannon Conjecture]%
\label{con:Stable_ENR-version_of_the_Cannon_Conjecture}
  Let $G$ be a torsion free hyperbolic group. Suppose that its
  boundary is homeomorphic to $S^{n-1}$.  Let $\Gamma$ be any $d$-dimensional Poincar\'e
  group for some natural number $d$ satisfying $n + d \ge 6$ which is a Farrell-Jones group.

  Then there is a closed aspherical $\ENR$ homology manifold $X$ of dimension $n+d$ which
  has the DDP and satisfies $\pi_1(X) \cong G \times \Gamma$.
\end{theorem}
\begin{proof} We conclude that $G \times \Gamma$ is a Farrell-Jones group from
  Theorem~\ref{the:status_of_the_Full_Farrell-Jones_Conjecture}~%
\eqref{the:status_of_the_Full_Farrell-Jones_Conjecture:Classes_of_groups:hyperbolic_groups}
  and~\eqref{the:status_of_the_Full_Farrell-Jones_Conjecture:inheritance:Passing_to_finite_direct_products}. 
 Since   $G$ is a Poincar\'e duality group of dimension $n$
  by~\cite[Corollary~1.3]{Bestvina-Mess(1991)}, the product $G \times \Gamma$ is a
  Poincar\'e duality group of dimension $n + d$. Since by assumption $n + d \ge 6$, we
  can apply Theorem~\ref{the:FJ_and_Borel-existence}.
\end{proof}


\section{Short review of Quinn's obstruction}
\label{sec:Short_review_of_Quinns_obstruction}

In order to replace $\ENR$ homology manifolds by topological
manifolds in the above result, we will use the following result that combines
work of Edwards and Quinn,  see~\cite[Theorems~3 and~4 on page~288]{Daverman(1986)}, 
\cite{Quinn(1987_resolution)}.

\begin{theorem}[Quinn's obstruction]\label{the:Quinn-obstruction}
  Let $X$ be a connected $\ENR$ homology manifold. There is an invariant $\iota(X) \in 1 + 8 \IZ$, known as the \emph{Quinn obstruction}, with the following properties:
  
  \begin{enumerate}
  \item\label{the:Quinn-obstruction:local}
        If $U \subset X$ is a connected non-empty  open subset, then 
        $\iota(U) = \iota(X)$;
  \item\label{the:Quinn-obstruction:manifold}
        Let $X$ be an $\ENR$ homology manifold of dimension $\geq 5$.
        Then the following are equivalent:
        \begin{itemize}
        \item $X$ has the DDP and $\iota(X) = 1$;
        \item $X$ is a topological manifold.
        \end{itemize}        
  \end{enumerate}
\end{theorem}

The elementary proof of the following result can be found in~\cite[Corollary~1.6]{Bartels-Lueck-Weinberger(2010)}.

\begin{lemma}\label{lem:Quinn_obstruction_and_boundary}
  Let $X$ be a connected $\ENR$ homology manifold with boundary $\partial X$.
  If $\partial  X$ is a manifold and $\dim(X) \ge 5$, then $\iota(\inti X) = 1$.  
\end{lemma}

Although we do not need the next result in this paper,
we mention that it follows from~\cite[Proposition~25.8 on page~293]{Ranicki(1992)} using 
Theorem~\ref{the:symmetric_signature_determines_symmetric_total_surgery_obstruction},
since we assume aspherical. 

\begin{theorem}[Relating the total surgery obstruction and Quinn's obstruction]%
\label{the:Relating_the_total_surgery_obstruction_and_Quinn's_obstruction}
  Let $B$ be an aspherical finite $n$-dimensional Poincar\'e complex which is homotopy
  equivalent to an $n$-dimensional closed $\ENR$ homology manifold $X$. Suppose that $\pi_1(B)$ is a
  Farrell-Jones group.

   Then we get
  \[
    i(X) = 8 \cdot h_0(s(B)) +1  = h^0(s^{\sym}(B)) +1.
    \]
  \end{theorem}

  Notice that in the situation of
  Theorem~\ref{the:Relating_the_total_surgery_obstruction_and_Quinn's_obstruction} the
  total surgery obstruction $s(B)$ is defined without the assumption that $B$ is homotopy
  equivalent to an $n$-dimensional closed $\ENR$ homology manifold and therefore does make sense
  for any aspherical $3$-dimensional Poincar\'e complex, and moreover, that
  $s(B)$ is a homotopy invariant, see Theorem~\ref{the:s(X)_is_homotopy_invariant}.

\begin{remark}\label{rem_Quinn_for_aspherical}
There is no example in the literature of a closed spherical $\ENR$ homology manifold
which is not homotopy equivalent to a closed topological manifold. 
\end{remark}


\section{$Z$-sets}
\label{sec:Z-sets}

\begin{definition}[Z-set]\label{def:Z-set}
  A closed subset $Z$ of a compact $\ENR$ $X$ is called a \emph{$Z$-set or a set of
    infinite deficiency} if for every open subset $U$ of $X$ the inclusion
  $U \setminus (U \cap Z) \to U$ is a homotopy equivalence.
\end{definition}

Any closed subset of the boundary $\partial M$ of a compact topological manifold $M$ is a
$Z$-set in $M$.  According to~\cite[page~470]{Bestvina-Mess(1991)} each of the following
properties characterizes $Z$-sets. Here, $X$ is a compact metric $\ENR$. The noncompact case is similar, except that maps and homotopies are limited by arbitrary open covers rather than by fixed $\epsilon\textrm{'s}$.

\begin{enumerate}

\item For every $\epsilon > 0$ there is a map $X \to X \setminus Z$ which is
  $\epsilon$-close to the identity.

\item For every closed subset $A \subseteq Z$, there exists a homotopy
  $H \colon X \times [0,1] \to X$ such that $H_0 = \id_X$, $H_t|_A$ is the inclusion
  $A \to X$ and $H_t(X \setminus A) \subseteq X \setminus Z$ for all $t > 0$.

\noindent To this, we will add:

\item  There exists a homotopy
  $H \colon X \times [0,1] \to X$ such that $H_0 = \id_X$ and $H_t(X ) \subseteq X \setminus Z$ for all $t > 0$.
\end{enumerate}

This last is (2) with $A=\emptyset$.  Clearly, (3) implies (1) and (2) implies (3), so (3) also suffices as a definition of $Z$-set. This is the definition we will use in what follows. Condition (3) implies that for every open $U \subset X$ the inclusion $U\setminus Z \to U$ is a homotopy equivalence. If $\alpha:S^k \to U$ is a map, then $H_t \circ \alpha :S^k \to X$ is a map homotopic to $\alpha$ and for $t>0$ its image lies in $X\setminus Z$. For $t>0$ sufficiently small, this homotopy takes place in $U$, so the inclusion-induced map $\pi_k(U \setminus Z) \to \pi_k(U)$ is surjective. A similar argument shows that $\pi_k(U \setminus Z) \to \pi_k(U)$ is a monomorphism -- if $\alpha$ extends over a disk in $U$, push the disk off of Z. The homotopy equivalence follows from the Whitehead Theorem, since $\ENR's$ have the homotopy types of CW complexes.
The next result is taken from~\cite[Proposition~2.5]{Bartels-Lueck-Weinberger(2010)}.

\begin{lemma}\label{lem:Z-set-yields-homology-mfd-with-boundary}
  Let $X$ be an $\ENR$ 
  which is the disjoint union of an $n$-dimensional 
  $\ENR$ homology manifold $\inti X$ and an $(n-1)$-dimensional 
  $\ENR$ homology manifold $\partial  X$ such that $\partial  X$ is a 
  $Z$-set in $X$.
  Then $X$ is an $\ENR$ homology manifold with boundary $\partial X$.
\end{lemma}

\begin{definition}[Compact sets become small at infinity]\label{def:Compact_sets_become_small_at_infinity}
  Consider a pair $(\overline{Y},Y)$ of $G$-spaces, $G$ a discrete group. We say that \emph{compact subsets of
    $Y$ become small at infinity}, if, for every $y \in \partial Y := \overline{Y}
  \setminus Y$, open neighborhood $U \subseteq \overline{Y}$ of $y$, and compact subset $K
  \subseteq Y$, there exists an open neighborhood $V \subseteq U$ of $y$ with
  the property that for every $g \in G$ we have the implication $g \cdot K \cap V \not= \emptyset
    \implies g \cdot K \subseteq U$.
\end{definition}

In the sequel we will choose $l$ large enough such that the following claims are true for
the torsion free hyperbolic group $G$  and its Rips complex $P_l(G)$.

\begin{enumerate}

\item The projection $P_l(G) \to P_l(G)/G$ is a model for the universal principal $G$-bundle
$EG \to BG$ and $P_l(G)/G$ is a finite $CW$-complex;

\item One can construct a compact topological space $\partial G$ and a compactification
  $\overline{P_l(G)}$ of $P_l(G)$ such that
  $\partial G = \overline{P_l(G)} \setminus P_l(G)$ holds, and $P_l(G)$ is open and dense in $\overline{P_l(G)}$;

\item $\overline{P_l(G)}$ is a compact metrizable $\ENR$ such that
  $\partial G \subset \overline{P_l(G)}$ is $Z$-set and $\overline{P_l(G)}$ has finite topological dimension;

\item Compact subsets of $Y$ become small at infinity for the pair $(\overline{P_l(G)},P_l(G))$.

\end{enumerate}

The first claim is proved for instance in~\cite{Meintrup-Schick(2002)}.
The second claim follows 
from~\cite[III.H.3.6 on page~429, III.H.3.7(3) and~(4) on page~430, III.H.3.7(4) on page~430 
and III.H.3.18(4) on page~433]{Bridson-Haefliger(1999)}
and~\cite[9.3.(ii)]{Bartels-Lueck-Reich(2008cover)}.
The third claim is due to Bestvina-Mess~\cite[Theorem~1.2]{Bestvina-Mess(1991)}, see 
also~\cite[Theorem~3.7]{Rosenthal-Schuetz(2005)}.
The fourth assertion is for instance proved in~\cite[page~531]{Rosenthal-Schuetz(2005)}.


\section{Pulling back boundaries}
\label{sec:pulling_back_boundaries}

We will need the following construction which may be interesting in its own right.

Let $(\overline{Y},Y)$ be a topological pair. Put $\partial Y := \overline{Y} \setminus Y$. 
Let $X$ be a topological space and $f \colon X \to Y$ be a continuous map. Define a
topological pair $(\overline{X},X)$ and a map $\overline{f} \colon \overline{X} \to X$, which will turn out to be continuous,
as follows.  The underlying set of $\overline{X}$ is the disjoint union $X\amalg  \partial Y$.  
We define the map of sets $\overline{f} \colon \overline{X} \to \overline {Y}$ to be $f \cup \id_{\partial Y}$.  
A subset $W$ of $\overline{X}$ is declared to be open if there exist open subsets $U \subseteq \overline{Y}$ 
and $V \subseteq X$ such that $W = \overline{f}^{-1}(U) \cup V$. We will see that this defines a topology.
Obviously $\overline{X}$ and $\emptyset$ are open. Given a collection of open subsets
$\{W_i \mid i \in I\}$, their union is again open by the following equality, if we write
$W_i = \overline{f}^{-1}(U_i) \cup V_i$ for open subsets $U_i \subset \overline{Y}$ and
$V_i \subseteq X$ and define open subsets $U := \bigcup_{i \in I} U_i \subseteq
\overline{Y}$ and $V := \bigcup_{i \in I} V_i\subseteq X$:
\begin{multline*}
\bigcup_{i \in I} W_i 
=  
\bigcup_{i \in I} \left(\overline{f}^{-1}(U_i) \cup V_i\right)
 = 
\bigcup_{i \in I} \overline{f}^{-1}(U_i) \cup \bigcup_{i \in I}  V_i
\\
= 
\overline{f}^{-1}\left(\bigcup_{i \in I} U_i\right) \cup \bigcup_{i \in I}  V_i
= 
\overline{f}^{-1}(U) \cup V.
\end{multline*}
Given two open subsets $W_1$ and $W_2$, their intersection is again open by the following
equality, if we write $W_i = \overline{f}^{-1}(U_i) \cup V_i$ for open subsets 
$U_i \subset \overline{Y}$ and $V_i \subseteq X$ for $i = 1,2$ and define open subsets 
$U := U_1 \cap U_2 \subseteq \overline{Y}$ and 
$V := \bigl(f^{-1}(U_1\cap Y) \cap V_2\bigr) \cup \bigl(V_1 \cap f^{-1}(U_2 \cap Y)\bigr) \cup \bigl(V_1 \cap V_2)
\subseteq X$:
\begin{eqnarray*}
\lefteqn{W_1 \cap W_2}
& & 
\\
& =  & 
 \bigl(\overline{f}^{-1}(U_1) \cup V_1\bigr) \cap \bigl(\overline{f}^{-1}(U_2) \cup V_2\bigr) 
\\
& = & 
\bigl(\overline{f}^{-1}(U_1) \cap \overline{f}^{-1}(U_2) \bigr) \cup \bigl(\overline{f}^{-1}(U_1) \cap  V_2\bigr)
\cup \bigl(V_1 \cap \overline{f}^{-1}(U_2)\bigr) \cup \bigl(V_1 \cap V_2\bigr)
\\
& = & 
\overline{f}^{-1}(U_1 \cap U_2) \cup
\left(\bigl(f^{-1}(U_1\cap Y) \cap  V_2\bigr)
\cup \bigl(V_1 \cap f^{-1}(U_2 \cap Y)\bigr) \cup \bigl(V_1 \cap V_2\bigr)\right)
\\
& = & 
\overline{f}^{-1}(U) \cup V.
\end{eqnarray*}

\begin{definition}[Pulling back the boundary]\label{def:pulling_back_the_boundary}
We say that $(\overline{f},f) \colon (\overline{X},X) \to (\overline{Y},Y)$ is 
obtained from $(\overline{Y},Y)$ by \emph{pulling back the boundary with $f$}.
\end{definition}

Notice that this is the smallest topology on the set $\overline{X} = X \amalg \partial Y$ 
for which $\overline{f}$ is continuous and $X \subseteq \overline{X}$ is an open subset.
This leads to the following universal property of the construction ``pulling back the boundary''.

\begin{lemma}\label{lem:universal_property_of_pulling_back_boundary} Let
  $(\overline{Y},Y)$ be a topological pair.  Let $X$ be a topological space and 
  $f \colon   X \to Y$ be a continuous map.  Suppose that 
  $(\overline{f},f) \colon (\overline{X},X)   \to (\overline{Y},Y)$ is obtained from $(\overline{Y},Y)$ 
  by pulling back the boundary   with $f$. Consider any pair of spaces $(\overline{\overline{X}},X)$ 
  and map of pairs
  $(\overline{\overline{f}},f) \colon (\overline{\overline{X}},X) \to (\overline{Y},Y)$
  such that $X$ is an open subset of $\overline{\overline{X}}$ and $\overline{\overline{f}}$ 
  induces a map $\overline{\overline{X}} \setminus X  \to \partial Y := \overline{Y} \setminus Y$.

  Then there is precisely one map $u \colon \overline{\overline{X}} \to \overline{X}$
  which induces the identity on $X$ and satisfies $\overline{f} \circ u =  \overline{\overline{f}}$.
\end{lemma}
\begin{proof}
  As a map of sets $u$ exists and is uniquely determined by the properties that $u$
  induces the identity on $X$ and $\overline{f} \circ u =
  \overline{\overline{f}}$. Namely, for $x \in X$ define $u(x) = x$ and for $x \in
  \overline{\overline{X}} \setminus X$ define $u(x)$ by 
$\overline{\overline{f}}(x) \in \partial \overline{Y} = \partial \overline{X} \subseteq \overline{X}$. We
  have to show that $u$ is continuous, i.e., $u^{-1}(W) \subseteq \overline{\overline{X}}$
  is open for every open subset $W \subseteq \overline{X}$.  By definition there are open
  subsets $U \subseteq \overline{Y}$ and $V \subseteq X$ such that $W =
  \overline{f}^{-1}(U) \cup V$.  Then $u^{-1}(W) = \overline{\overline{f}}^{-1}(U) \cup
  V$. Since $\overline{\overline{f}}$ is continuous, $\overline{\overline{f}}^{-1}(U)
  \subseteq \overline{\overline{X}}$ is open. Since $X$ is open in
  $\overline{\overline{X}}$ and the topology on $X$ is the subspace topology of $X
  \subseteq \overline{\overline{X}}$, we conclude that for any open subset $V \subseteq X$ the subset $V
  \subseteq \overline{\overline{X}}$ is open. Hence $u^{-1}(W) \subseteq
  \overline{\overline{X}}$ is open.
\end{proof}

\begin{lemma}
\label{lem:properties_of_pulling_back_the_boundary}

Let $(\overline{Y},Y)$ be a topological pair.  Let $X$ be a topological space and
$f \colon X \to Y$ be a continuous map.  Suppose that
$(\overline{f},f) \colon (\overline{X},X) \to (\overline{Y},Y)$ is obtained from
$(\overline{Y},Y)$ by pulling back the boundary with $f$.

\begin{enumerate}

\item\label{lem:properties_of_pulling_back_the_boundary:X_dense} If
  $Y \subseteq \overline{Y}$ is dense and the closure of the image of $f$ in
  $\overline{Y}$ contains $\partial Y$, then $X \subseteq \overline{X}$ is dense;

\item\label{lem:properties_of_pulling_back_the_boundary:overline(X)-compact} Suppose that 
  $\overline{Y}$ is compact, $Y \subseteq \overline{Y}$ is open and $f \colon X \to Y$ is
  proper.  Then $\overline{X}$ is compact;

\item\label{lem:properties_of_pulling_back_the_boundary:dim}
We have for the topological dimension of $\overline{X}$
\[
\dim(\overline{X}) \le \dim(X) + \dim(\overline{Y}) +1;
\]

\item\label{lem:properties_of_pulling_back_the_boundary:overline(f)-continuous}
The map $\overline{f} \colon \overline{X} \to \overline{Y}$ given by $f \cup \id_{\partial Y}$ is continuous;

\item\label{lem:properties_of_pulling_back_the_boundary:homeo_on_boundary}
The induced map $\overline{f}$ induces a homeomorphism $\partial f \colon \partial {X} \to \partial{Y}$;

\item\label{lem:properties_of_pulling_back_the_boundary:functoriality} Let
  $g \colon Z \to X$ be a map.  Suppose that
  $(\overline{f},f) \colon (\overline{X},X) \to (\overline{Y},Y)$ and
  $(\overline{f \circ g},f \circ g) \colon (\overline{Z},Z) \to (\overline{Y},Y)$
  respectively are obtained by pulling back the boundary of $(\overline{Y},Y)$ with $f$
  and $f \circ g$ respectively.  Let
  $\overline{g} \colon (\overline{\overline{Z}},Z) \to (\overline{X},X)$ be obtained by
  pulling back the boundary of $(\overline{X},X)$ with $g$.

Then we get an equality of topological spaces $\overline{\overline{Z}} = \overline{Z}$
and of maps $\overline{f} \circ \overline{g} = \overline{f \circ g}$.

\end{enumerate}
\end{lemma}
\begin{proof}~\eqref{lem:properties_of_pulling_back_the_boundary:X_dense} Consider
  $x \in \partial X$ and a neighborhood $W$ of $x$ in $\overline{X}$. We have to show
  $X \cap W \not= \emptyset$.  We can write $W = \overline{f}^{-1}(U) \cup V$ for open
  subsets $U \subset \overline{Y}$ and $V \subseteq X$.  Without loss of generality we can
  assume $V = \emptyset$, or, equivalently $W = \overline{f}^{-1}(U)$ for open subset
  $U \subset \overline{Y}$. Obviously $U$ is an open neighborhood of $\overline{f}(x) \in \overline{Y}$.
  Since by assumption the closure of the image of $f$ in $\overline{Y}$ contains
  $\partial Y$, we have $\im(f) \cap U \not= \emptyset$ and hence
  $X \cap W \not= \emptyset$.
  \\[1mm]~\eqref{lem:properties_of_pulling_back_the_boundary:overline(X)-compact} Let
  $\{W_i \mid i \in I\}$ be an open covering of $\overline{X}$.  We can write
  $W_i = \overline{f}^{-1}(U_i) \cup V_i$ for open subsets $U_i \subset \overline{Y}$ and
  $V_i \subseteq X$. Then $\{U_i \cap \partial Y \mid i \in I\}$ is an open covering of
  $\partial Y$.  Since $\partial Y \subseteq Y$ is closed and $\overline{Y}$ is compact by
  assumption, $\partial Y$ is compact.  Hence there is a finite subset $J \subseteq I$
  with $\partial Y \subseteq \bigcup_{i \in J} U_i$. The set
  $\overline{Y} \setminus \bigl(\bigcup_{i \in J} U_i\bigr)$ is closed in $\overline{Y}$
  and hence compact.  Since $\overline{Y} \setminus \bigl(\bigcup_{i \in J} U_i\bigr)$ is
  contained in $Y$ and $f \colon X \to Y$ is by assumption proper, the preimage
  $f^{-1}\left(\overline{Y} \setminus \bigl(\bigcup_{i \in J} U_i\bigr)\right)$ is also
  compact.  Hence there is a finite subset $J' \subseteq I$ such that
  $\{V_i \mid j \in J'\}$ covers
  $f^{-1}\left(\overline{Y} \setminus \bigl(\bigcup_{i \in J} U_i\bigr)\right)$. Hence
  $\{W_i \mid i \in J \cup J'\}$ covers $\overline{X}$. This shows that $\overline{X}$ is
  compact.  
\\[1mm]~\eqref{lem:properties_of_pulling_back_the_boundary:dim} Consider any
  open covering $\calw = \{W_i \mid i \in I\}$ of $\overline{X}$. By definition there are
  $U_i \subseteq \overline{Y}$ and $V_i \subseteq X$ such that
  $W_i = \overline{f}^{-1}(U_i) \cup V_i$.  Now put
  \begin{eqnarray*}
  \calw_{\partial X} & := & \{\overline{f}^{-1}(U_i) \mid i \in I\};
  \\
   \calw_X & := & \{W_i \cap X  \mid i \in I\}.
 \end{eqnarray*}
 Then $\calw_{\partial X}  \cup \calw_X$ is an open covering of $\overline{X}$, which is a refinement
 of $\calw$.  Moreover, $\calw_X$ is an open covering of $X$ and the union of the elements
 in $\calw_{\partial X}$  contains $\partial X$.  We can find an open covering
 $\calv_X$ whose covering dimension is less or equal to $\dim(X)$ and which refines $\calw_X$.  We
 obtain an open covering $\{U_i \mid i \in I\} \cup \{Y\}$ of $\overline{Y}$, since
 $\partial Y$ is contained in $\bigcup_{i \in I} U_i$. We can find an open covering
 $\calv_{\overline{Y}}$ of $\overline{Y}$ which is a refinement of $\{U_i \mid i \in I\}
 \cup \{Y\}$ and has dimension $\le \dim(\overline{Y})$.  Put
  \[
   \calv_{\partial Y} := \{V \in \calv_{\overline{Y}} \mid V \cap \partial Y \not= \emptyset\}.
  \]
  Then $ \calv_{\partial Y}$ is a refinement of $\{U_i \mid i \in I\}$, has covering
  dimension $\le \dim(\overline{Y})$ and the union of the elements in $ \calv_{\partial
    Y}$ contains $\partial Y$. Define $\overline{f}^*\calv_{\partial X}$ to be the
  collection of open subsets of $\overline{X}$ given by $\{\overline{f}^{-1}(V) \mid V \in
  \calv_{\partial Y}\}$. Then $\overline{f}^*\calv_{\partial X}$ is a refinement of
$\calw_{\partial X}$, has covering  dimension  $\le \dim(\overline{Y})$ and the union of its
elements contains $\partial X = \partial Y$. Put
\[
\calv = \calw_{X} \cup \overline{f}^*\calv_{\partial X}.
\]
Then $\calv$ is an open covering of $\overline{X}$ which refines $\calw$. Its covering
dimension satisfies
\[
\dim(\calv) \le  \dim(\calv_{X}) + \dim(\overline{f}^*\calv_{\partial X}) + 1 \le \dim(X) + \dim(\overline{Y}) +1.
\]
\eqref{lem:properties_of_pulling_back_the_boundary:overline(f)-continuous} If $U
  \subseteq \overline{Y}$ is open, then by definition $\overline{f}^{-1}(U)\subseteq
  \overline{X}$ is open.
  \\[1mm]~\eqref{lem:properties_of_pulling_back_the_boundary:homeo_on_boundary} Obviously
  $\overline{f} \colon \overline{X} \to \overline{Y}$ induces a bijective continuous map
  $\partial f \colon \partial X \to \partial Y$. We have to show that it is open.  An open
  subset of $\partial X$ is of the form $\bigl(\overline{f}^{-1}(U) \cup V
  \bigr)\cap \partial X$ for some open subsets $U \subseteq \overline{Y}$ and $V \subseteq  X$. 
  Its image under $\partial f$ is $U \cap \partial Y$ and hence an open subset of $\partial Y$.
  \\[1mm]~\eqref{lem:properties_of_pulling_back_the_boundary:functoriality}. Notice that
  as sets $\overline{\overline{Z}}$ and $\overline{Z}$ agree, both look like 
  $Z  \amalg \partial Y$.  Next we show that the two topologies agree. A subset $W$ of
  $\overline{\overline{Z}}$ is open if there are open subsets $U \subseteq \overline{Y}$
  and $V_2 \subseteq Z$ with $W = \overline{f \circ g}^{-1}(U) \cup V_2$. A subset $W_1
  \subseteq \overline{X}$ is open if there exist open subsets $U \subseteq \overline{Y}$
  and $V_1 \subseteq X$ with $W_1 = \overline{f}^{-1}(U) \cup V_1$. A subset $W_2$ of
  $\overline{Z}$ is open, if there exist open subsets $W_1 \subseteq \overline{X}$ and
  $V_2 \subseteq Z$ such that $W_2$ looks like $\overline{g}^{-1}(W_1) \cup V_2$. This is
  equivalent to the existence of open subsets $U \subseteq \overline{Y}$, $V_1 \subseteq
  X$ and $V_2 \subseteq Z$ such that
\[
W_2 = \overline{g}^{-1}\bigl(\overline{f}^{-1}(U) \cup V_1) \cup V_2.
\]
Since
\[\overline{g}^{-1}\bigl(\overline{f}^{-1}(U) \cup V_1) \cup V_2 = 
\overline{f \circ g}^{-1}(U) \cup \left( g^{-1}(V_1) \cup V_2\right)
\]
and $g^{-1}(V_1) \cup V_2$ is an open subset of $Z$,
the topology on  $\overline{\overline{Z}}$ is finer than the  topology on $\overline{Z}$.
So it remains to show that the topology on $\overline{Z}$ is finer than the topology on 
$\overline{\overline{Z}}$. This follows from the observation that 
for  open subsets $U \subseteq \overline{Y}$ and  $V_2 \subseteq Z$  we get
\[
\overline{f \circ g}^{-1}(U) \cup V_2 = \overline{g}^{-1}(\overline{f}^{-1}(U) \cup \emptyset) \cup V_2.
\]
\end{proof}

\begin{example}[One-point-compactification]\label{exa:one-point-compactification}
  Let $X$ and $Y$ be locally compact Hausdorff spaces. Denote by $X^c$ and $Y^c$
  their one-point-compactification.  Let $f \colon X \to Y$ be a map. Denote by
  $(\overline{X},X)$ the space obtained from $(Y^c,Y)$ by pulling back the boundary with
  $f$.  

  Consider first the case where $f$ is proper. Recall that a subset $W \subseteq Y^c = Y  \cup \{\infty\}$ 
  is open if it belongs to $Y$ and is open in $Y$ or there is a compact
  subset $C \subseteq Y$ such that $W = Y^c \setminus C$. This is indeed a topology,
  see~\cite[page~184]{Munkres(1975)}. By construction the underlying sets for
  $\overline{X}$ and $X^c$ agree, namely, they are both given by $X \amalg
  \{\infty\}$. Next we compare the topologies.
  
  Consider an open subset $W$ of $\overline{X}$. We want to show that $W \subseteq X^c$ is
  open.  We can write $W = \overline{f}^{-1}(U) \cup V$ for open subsets $U \subseteq Y^c$
  and $V \subseteq X$. If $\infty$ does not belong to $U$, then $U$ is an already open
  subset of $Y$ and $\overline{f}^{-1}(U) = f^{-1}(U)$ is an open subset of $X$ which
  implies that $W \subseteq X$ and hence $W \subseteq X^c$ are  open. It remains to treat the case $\infty \in U$.
  From the definitions we conclude that we can write
  $W = \overline{f}^{-1}(Y^c \setminus C) \cup V$ for some compact subset $C \subseteq Y$
  and an open subset $V$ of $Y$. Since
  \[
  \overline{f}^{-1}(Y^c \setminus C) = \overline{X} \setminus f^{-1}(C)
  \]
  and by the properness of $f$ the set $f^{-1}(C) \subseteq X$ is compact, $W$ is open
  regarded as a subset of $X^c$.  

  This shows that  the identity induces a continuous bijective map
  $X^c \to \overline{X}$.  (One can also deduce this directly 
  from Lemma~\ref{lem:universal_property_of_pulling_back_boundary}.)

Since $X^c$ is compact and $\overline{X}$ is Hausdorff, this is
  a homeomorphism, see~\cite[Theorem~5.6 in Chapter~III on page~167]{Munkres(1975)}.
  Hence we get an equality of topological spaces $\overline{X} = X^c$ and of maps
  $\overline{f} = f^c$. 

  Now consider the case where $f$ is the constant map onto some point $y_0 \in Y$. Suppose
  that $X$ is not compact, or, equivalently, that the constant map $f$ is not proper. The set
  $Y^c \setminus \{y_0\}$ is open in $Y^c$. Hence 
  $\partial X = \{\infty\} = \overline{f}^{-1}(Y^c   \setminus \{y_0\}) $ is an open subset of $\overline{X}$. 
  Since also $X \subseteq   \overline{X}$ is open, $\overline{X}$ is, as a topological space, the disjoint 
  union $X  \amalg \{\infty\}$. Since $X$ is not compact, its one-point compactification is not
  homeomorphic to $\overline{X}$.
\end{example}

\begin{remark}[Dependency on $f$]\label{def:Dependency_on_f}
  Example~\ref{exa:one-point-compactification} shows that $\overline{X}$ does depend on
  the choice of $f$.  So the reader should be careful when we just write $\overline{X}$
  without including $f$ in the notation. 
\end{remark}

\begin{lemma}\label{lem:small_at_infinite_pulling_back_boundaries}
  Consider a pair $(\overline{Y},Y)$ of $G$-spaces, $G$ a discrete group,  such that compact
  subsets of $Y$ become small at infinity in the sense of
  Definition~\ref{def:Compact_sets_become_small_at_infinity}.  Let
  $f \colon X \to Y$ be a $G$-map. Suppose that $(\overline{X},X)$ is
  obtained from $(\overline{Y},Y)$ by pulling back the boundary with
  $f$.

  Then compact subsets of $X$ become small at infinity.
\end{lemma}
\begin{proof}
  Consider an element $x \in \partial X$, an open neighborhood $U \subseteq \overline{X}$
  of $x$, and a compact subset $K \subseteq X$. We can find an open neighborhood $U'
  \subseteq Y$ of $f(x)  \in \partial Y$ and an open subset $W \subseteq X$ such
  that $U = \overline{f}^{-1}(U') \cup W$. Put $L = f(K)$.  Then $L \subseteq Y$ is
  compact. By assumption we can find an open neighborhood $V' \subseteq U'$ of $f(x)
  \in \partial Y$ such that the implication $g \cdot L \cap V'
  \not= \emptyset \implies g \cdot L \subseteq U'$ holds for every $g \in G$. Put $V =
  \overline{f}^{-1}(V')$. This is an open neighborhood of $x \in \partial X$ with $V
  \subseteq U$.  Moreover we get for every $g \in G$
  \begin{multline*}
    g \cdot K \cap V \not= \emptyset \implies g \cdot L \cap V' \not= \emptyset
    \\
    \implies g \cdot L \subseteq U' \implies g \cdot f^{-1}(L) \subseteq \overline{f}^{-1}(U')
    \implies g \cdot K \subseteq U.
  \end{multline*}
\end{proof}

\begin{definition} [Continuously controlled over $Y$ at $\partial Y$]\label{def:continuously_controlled} Consider a
  pair $(\overline{Y},Y)$ of spaces and a homotopy equivalence $f \colon X \to Y$.  We
  call $f$ \emph{continuously controlled over $Y$ at $\partial Y$} if there exists a map $u \colon Y \to X$ and homotopies
  $h \colon f \circ u \simeq \id_Y$ and $k \colon u \circ f \simeq \id_X$ with the
  following property: For every $z \in \partial Y = \overline{Y} \setminus Y$ and
  neighborhood $U$ of $z$ in $\overline{Y}$ there is an open neighborhood $V$ of $z$ in
  $\overline{Y}$ with $V \subseteq U$ such that the following two implications are true:
  \begin{itemize}
  \item $y \in V \implies h(\{y\} \times [0,1]) \subseteq U$;
   \item $x \in f^{-1}(V) \implies f \circ k(\{x\} \times [0,1]) \subseteq U$.
   \end{itemize}
 \end{definition}

\begin{lemma}\label{lem:criterion_for_proper}
Let $f \colon X \to Y$ be a $G$-map of proper free  $G$-spaces, $G$ a discrete group. Suppose that $X$ is cocompact.
Then $f$ is proper.
\end{lemma}
\begin{proof}
We have the following pullback
\[
\xymatrix{X  \ar[d] \ar[r]^f & Y \ar[d]
\\
X/G \ar[r]^{f/G} & Y/G
}
\]
where the vertical maps are principal $G$-bundles. Since $X/G$ is  compact,
$f/G$ is proper. Hence $f$ is proper by~\cite[Lemma~1.16 on page~14]{Lueck(1989)}.
\end{proof}

\begin{lemma}\label{lem:pulling_back_Z-sets} Consider a pair $(\overline{Y},Y)$ of spaces
  such that $\overline{Y}$ is a compact $\ENR$ and $\partial Y$ is a $Z$-set in
  $\overline{Y}$. Consider a homotopy equivalence $f \colon X \to Y$ which is continuously
  controlled. Let $(\overline{f},f) \colon (\overline{X},X) \to (\overline{Y},Y)$ be
  obtained by pulling back the boundary along $f$.

  Then $\overline{X}$ is an $\ENR$ and $\partial X \subseteq \overline{X}$ is a Z-set.
\end{lemma}
\begin{proof} We will use the third characterization of $Z$-set from Definition \ref{def:Z-set}. This characterization says that if
  $\overline{X}=X \cup \partial X$ with $X$ a compact $\ENR$ and if there is a homotopy
  $h_{t}:\overline{X} \to \overline{X}$ with $h_{0}=\id$ and
  $h_{t}(\overline{X}) \subset X$ for all $t>0$, then $\overline{X}$ is an $\ENR$ and
  $\partial X$ is a $Z$-set in $\overline{X}$.

  The statement in~\cite{Bestvina(1996)} assumes that $\overline{X}$ is an $\ENR$, but this
  is unnecessary in connection with definition (3), since Hanner's criterion, see~\cite[Theorem~7.2]{Hanner(1951)}, says that a
  compact metric space is an $\ENR$ if it is $\epsilon$-dominated by $\ENR$s for every
  $\epsilon > 0$. The homotopy $h_{t}$ above shows that the $\ENR$ $X$ $\epsilon$-dominates
  $\overline{X}$ for every $\epsilon>0$.\footnote{A map $f:X \to Y$ between metric spaces is an $\epsilon$-domination if there is a map $g:Y \to X$ so that the composition $f \circ g:Y \to Y$ is $\epsilon$-homotopic to the identity.}\footnote{Hanner's Theorem, as stated in 
  \cite{Hanner(1951)}, is much more general than the version we have stated here. Hanner's theorem applies to $\ANR's$, by which he means separable metric spaces $X$ such that whenever $X$ is imbedded as a closed subset of another separable metric space $Z$, it is a retract of some neighborhood in $Z$. In particular, $X$ need not be even locally compact. In order to achieve such generality, it is necessary to consider homotopy dominations limited by open covers rather than by fixed constants $\epsilon$.}

  Let $c_{t}:\overline{Y} \to \overline{Y}$ be a homotopy so that
  $c_{0}=\id_{\overline{Y}}$ and $c_{t}(\overline{Y}) \subset Y$ for all $t>0$. The
  homotopy equivalence $f$ has a homotopy inverse $g:Y \to X$. The continuous control
  condition means that $f$ extends continuously by the identity on
  $\partial X = \partial Y$ to $\bar f:\overline{X} \to \overline{Y}$, $g$ extends
  continuously by the identity to $\bar g:\overline{Y} \to \overline{X}$ and there are
  homotopies $h_{t}$ from $\id_{Y}$ to $f\circ g$ and $k_{t}$ from $\id_{X}$ to
  $g \circ f$ which extend continuously by the identity to $\bar h_{t}$ and $\bar k_{t}$. 
  Restricted to $X$ and $Y$, all of these maps and homotopies are proper.
 
  For $x \in \overline{X}$, let
  $\alpha(x)= \min(\diam(\{\bar k_{t}(x), 0\le t \le 1\}),\frac{1}{2})$. Set
  \[
  \bar e_{t}=
  \begin{dcases}
    \bar k_{t/\alpha(x)}(x) &0\le t \le \alpha(x),\ \alpha(x)\ne 0; \\
    \bar g\circ c_{t-\alpha(x)}\circ \bar f(x) & \alpha(x) \le t \le 1\ \mathrm{or}\
    \alpha(x)=0.
  \end{dcases} 
  \]
  For $t=0$ and $\alpha(x) \ne 0$, we have
  $\bar e_{0}(x) = \bar k_{0}(x)=x$. If $t=0$ and $\alpha(x)=0$, we have
  $\bar e_{0}(x)=\bar g\circ c_{0}\circ \bar f(x)=x$, since $\alpha(x)=0$ implies that
  $\bar k_{t}(x)=x$ for all $0 \le t \le 1$. If $t=\alpha(x) \ne 0$,
  $\bar e_{t}(x) = \bar g \circ \bar f(x)$ with either definition. If $t=\alpha(x)$, then
  both definitions give $\bar g \circ \bar f(x)$.
  This shows that $\bar e_t$ is a well-defined continuous function with $e_0=\id_{\overline X}$. For 
  any $x \in \partial X$, $\alpha(x)=0$ and $\bar e_t(x)=\bar g\circ c_t \circ \bar f(x) = 
  \bar g \circ  c_t(x)$. Since $c_t(x)\in Y$, $\bar e_t(x)=\bar g \circ  c_t(x) \in X$, as desired.  
The formula above shows that points of $X$ have no possibility of moving back into $\partial X$, so the proof is complete.
  \end{proof}

  \begin{lemma}\label{lem:getting_continuous_control} Consider a $G$-homotopy equivalence of
    $f \colon X \to Y$ of cocompact proper free $G$-spaces, $G$ a discrete group. Suppose that $Y$ is a subspace of the
    compact $G$-space $\overline{Y}$ such that compact subsets become small at
      infinity for $(\overline{Y},Y)$.

    Then $f$ is continuously controlled at $\partial Y$.
    \end{lemma}
    \begin{proof} Choose a $G$-map $u \colon Y \to X$ and $G$-homotopies $h \colon f \circ u \simeq \id_Y$ 
    and $k \colon u \circ f \simeq \id_X$.  Choose a compact subset
    $C \subseteq Y$ such that $G \cdot C = Y$ holds.

    Fix a point $z \in \partial Y = \overline{Y} \setminus Y$ and an open
    neighborhood $U$ of $z$ in $\overline{Y}$. 

    Since compact subsets become small at
    infinity for $(\overline{Y},Y)$, we can find an open neighborhood $V$ of $z$ in
    $\overline{Y}$ with $V \subseteq U$ such that for every $g \in G$ we have the
    implication
    $g \cdot h(C \times [0,1]) \cap V \not= \emptyset \implies g \cdot h(C \times [0,1])
    \subseteq U$.  

    Consider $y \in V$. We can find $g \in G$ with $y \in g\cdot C$.  Since
    $y = h(y,1) \in g\cdot h(C \times [0,1])$, we get
    $g \cdot h(C \times [0,1]) \cap V \not= \emptyset$.  This implies
    $g \cdot h(C \times [0,1]) \subseteq U$ and in particular
    $h(\{y\} \times [0,1]) \subseteq U$.

    The map $f$ is proper by Lemma~\ref{lem:criterion_for_proper}. Hence
    $f^{-1}(C) \subseteq X$ is compact. Let
    $(\overline{f},f) \colon (\overline {X},X) \to (\overline{Y},Y)$ be obtained by
    pulling back the boundary with $f$. Since compact subsets become small at infinity for
    $(\overline{X},X)$ by Lemma~\ref{lem:small_at_infinite_pulling_back_boundaries}, we
    can find an open neighborhood $V'$ of $z \in \partial X = \partial Y$ in $\overline{X}$ with
    $V' \subseteq \overline{f}^{-1}(U)$ such that for every $g \in G$ we have the
    implication
    $g \cdot k(f^{-1}(C) \times [0,1]) \cap V' \not= \emptyset \implies g\cdot k(f^{-1}(C)
    \times [0,1]) \subseteq \overline{f}^{-1}(U)$.  Choose an open subset $V'' \subseteq \overline{Y}$
    and an open subset $W \subseteq X$ with $V' = \overline{f}^{-1}(V'') \cup W$. Since the
    implication above remains true if we shrink $V'$, we can assume without loss of
    generality that $V' = \overline{f}^{-1}(V'')$. In particular $V''$ is an open neighborhood
    of $z \in \overline{Y}$.

    Consider $x \in X$ with $f(x) \in V''$.  Then $x \in \overline{f}^{-1}(V'')$. We can find
    $g \in G$ with $x \in g \cdot f^{-1}(C)$.  Since
    $x = k(x,1) \in g\cdot k(f^{-1}(C)\times [0,1])$, we get
    $g \cdot k(f^{-1}(C) \times [0,1]) \cap \overline{f}^{-1}(V'') \not= \emptyset$.  This
    implies $g \cdot k(f^{-1}(C) \times [0,1]) \subseteq \overline{f}^{-1}(U)$ and hence
    $f \circ k(\{x\} \times [0,1]) \subseteq U$.
    
  \end{proof}


\section{Recognizing the structure of a manifold with boundary}
\label{sec:Recognizing_the_structure_of_a_manifold_with_boundary}

Recall that we have discussed the basic properties of the Rips complex $P_l(G)$ before in 
Section~\ref{sec:Z-sets}.

\begin{theorem}\label{the:adding_boundary}
  Let $G$ be torsion free hyperbolic group $G$ with boundary $S^2$.  Consider a homotopy
  equivalence $f \colon M \to P_l(G)/G \times N$, where $M$ is a closed homology
  $\ENR$-manifold, and $N$ is a closed topological manifold of dimension $\ge 2$. Denote by
  $p_G \colon P_l(G) \to P_l(G)/G$ the canonical projection.  Let the $G$-covering $\widehat{M} \to M$
  be the pullback with $f$ of the  $G$-covering
  $p_G \times \id_N \colon P_l(G) \times N \to  P_l(G)/G \times N$ and
  $\widehat{f}\colon \widehat{M} \to P_l(G) \times N$ be the induced $G$-homotopy
  equivalence. Let
  $(\overline{\widehat{f}}, \widehat{f}) \colon
  (\overline{\widehat{M}},\widehat{M}) \to (\overline{P_l(G)},P_l(G))$ be obtained
 by pulling back the boundary along $\widehat{f}$.

Then $\overline{\widehat{M}}$ is a compact $\ENR$ homology manifold 
whose boundary $\partial \overline{\widehat{M}}$ is $S^2 \times N$ and a $Z$-set. 
\end{theorem}
\begin{proof}
Recall from Section~\ref{sec:Z-sets} that $P_l(G) \to P_l(G)/G$ is a model for the universal principal $G$-bundle
$EG \to BG$ and $P_l(G)/G$ is a finite $CW$-complex. Hence $P_l(G)$ is a cocompact free proper $G$-space.
Compact subsets of $P_l(G)$ become small at infinity for the pair $(\overline{P_l(G)},P_l(G))$.
The space $\overline{P_l(G)}$ is a compact metrizable $\ENR$  and 
$\partial P_l(G) \subseteq \overline{P_l(G)}$ is a $Z$-set. We conclude from
Lemma~\ref{lem:pulling_back_Z-sets}  and Lemma~\ref{lem:getting_continuous_control}  that
$\partial \overline{\widehat{M}} \subseteq \overline{\widehat{M}}$ is a $Z$-set and 
$\overline{\widehat{M}}$ is an $\ENR$.
We conclude that $\overline{\widehat{M}}$ is compact and has finite dimension from
Lemma~\ref{lem:properties_of_pulling_back_the_boundary}~%
\eqref{lem:properties_of_pulling_back_the_boundary:overline(X)-compact}
and~\eqref{lem:properties_of_pulling_back_the_boundary:dim},
and Lemma~\ref{lem:criterion_for_proper}.
Lemma~\ref{lem:Z-set-yields-homology-mfd-with-boundary}   implies that $\overline{\widehat{M}}$ 
is an $\ENR$ homology manifold with boundary
in the sense of Definition~\ref{def:homology-ENR-manifold_with_boundary}.
\end{proof}


\section{Proof of Theorem~\ref{the:Vanishing_of_the_surgery_obstruction} 
and Theorem~\ref{the:stable_Cannon_Conjecture}} 
\label{sec:Proof_of_the_main_theorem}

This section is entirely devoted to the proof of Theorem~\ref{the:Vanishing_of_the_surgery_obstruction} 
and Theorem~\ref{the:stable_Cannon_Conjecture}. We begin with the following considerations.

Consider a hyperbolic $3$-dimensional Poincar\'e duality group $G$.  Then $G$ is
torsion free and $\partial G$ is $S^2$ by Theorem~\ref{the_hyperbolic_boundary_sphere}.
Let $N$ be a closed aspherical topological manifold of dimension $n \ge 3$ with fundamental
group $\pi$. Suppose that $\pi$ is a Farrell-Jones group.  Then $G \times \pi$ is a
finitely presented $(3+n)$-dimensional Poincar\'e duality group.  We conclude that
$G \times \pi$ is a Farrell-Jones group from
Theorem~\ref{the:status_of_the_Full_Farrell-Jones_Conjecture}~%
\eqref{the:status_of_the_Full_Farrell-Jones_Conjecture:Classes_of_groups:hyperbolic_groups}
and~\eqref{the:status_of_the_Full_Farrell-Jones_Conjecture:inheritance:Passing_to_finite_direct_products}.
Since $3+n \ge 6$, we conclude from Theorem~\ref{the:FJ_and_Borel-existence} that there is
a closed $\ENR$ homology manifold $M$ having the DDP and a homotopy equivalence
$M \to BG \times N$.

Denote by $p_G \colon P_l(G) \to P_l(G)/G$ the canonical projection.  Let the $G$-covering
$\widehat{M} \to M$ be the pullback with $f$ of the $G$-covering
$p_G \times \id_N \colon P_l(G) \times N \to P_l(G)/G \times N$ and
$\widehat{f}\colon \widehat{M} \to P_l(G) \times N$ be the induced $G$-homotopy
equivalence. Let
$(\overline{\widehat{f}}, \widehat{f}) \colon (\overline{\widehat{M}},\widehat{M}) \to
(\overline{P_l(G)} \times N,P_l(G) \times N)$ be obtained by pulling back the boundary along $\widehat{f}$.
Theorem~\ref{the:adding_boundary} implies that $\overline{\widehat{M}}$ is a compact
$\ENR$ homology manifold whose boundary, $\partial \overline{\widehat{M}}$ is homeomorphic to $S^2 \times N$ and
is a $Z$-set in $\overline{\widehat{M}}$. 

We conclude from
Lemma~\ref{lem:Quinn_obstruction_and_boundary}
that $i(\widehat{M}) = 1$. Theorem~\ref{the:Quinn-obstruction}~\eqref{the:Quinn-obstruction:local} then implies 
that $i(M) = 1$. 

We conclude from Theorem~\ref{the:Quinn-obstruction}~\eqref{the:Quinn-obstruction:manifold} 
and a collaring result due to Ferry and Seebeck, which can be found in~\cite[Theorem~1 in Section~40 on page~285]{Daverman(1986)}, 
that $\overline{\widehat{M}}$ is a compact topological manifold with boundary
$\partial \overline{\widehat{M}} = S^2 \times N$. $M$ is a closed topological manifold since it has DDP and Quinn index 1.\footnote{ We note that a $Z$-set $Z$ in a compact $\ENR$ $W$ is automatically 1-LCC. $\ENR$s are locally simply connected, so for every $\epsilon > 0$ there is a $\delta >0$ so that every map $\alpha: S^1 \to W$ with diameter $< \delta$ extends to a map $\bar \alpha:D^2 \to W$ with diameter $< \epsilon$. The $Z$-set property allows us to push $\bar \alpha(D^2)$ off of $Z$ by an arbitrarily small homotopy, giving the desired extension. See the discussion following Definition \ref{def:Z-set} for further details.  }
)

Since $\partial P_l(G)$ is a $Z$-set in
$\overline{P_l(G)}$,  $\partial P_l(G) \times N$ is a $Z$-set in $\overline{P_l(G)} \times N$.
We know already that $\partial \overline{\widehat{M}}$ is a $Z$-set in $\overline{\widehat{M}}$.
Since $P_l(G)$ is contractible, $\overline{P_l(G)}$ is contractible.
Since $\widehat{f}$ is a homotopy equivalence, $\overline{\widehat{f}}$ is a homotopy equivalence.
By construction $\overline{\widehat{f}}$ induces a homeomorphism 
$u \colon \partial \overline{\widehat{M}} = \partial P_l(G) \times N = S^2 \times N$.
Since $\overline{P_l(G)}$ is contractible, we can find a homotopy equivalence
$\overline{P_l(G)} \to D^3$ which is the identity on $\partial P_l(G) = S^2$.
Hence there is a homotopy equivalence
$(U,u) \colon (\overline{\widehat{M}}, \partial \overline{\widehat{M}}) \to (D^3 \times N, S^2 \times N)$ 
such that $u$ is a homeomorphism.  Since $\pi_1(D^3 \times N) \cong \pi$ is a
Farrell-Jones group and $3+n \ge 6$, the relative Borel Conjecture\footnote{This is the Borel Conjecture for compact manifolds with boundary, rel boundary. The surgery exact sequence shows that it holds in the usual dimension range whenever the assembly map is an isomorphism.} holds, i.e., we can
change $(U,u)$ up to homotopy relative $\partial \overline{\widehat{M}}$ such
that we obtain a homeomorphism of pairs
\[
(U,u) \colon (\overline{\widehat{M}}, \partial \overline{\widehat{M}}) \to (D^3 \times N, S^2 \times N).
\]

\begin{proof}[Proof of Theorem~\ref{the:Vanishing_of_the_surgery_obstruction}] 
The considerations above  applied in the special case $N = T^3$
show that $BG \times T^3$ is homotopy equivalent to  a closed topological manifold,
since $\IZ^3$ is a Farrell-Jones group by 
Theorem~\ref{the:status_of_the_Full_Farrell-Jones_Conjecture}~%
\eqref{the:status_of_the_Full_Farrell-Jones_Conjecture:Classes_of_groups:CAT(0)-groups}.
Hence $s^{\sym}(BG \times T^3)$ vanishes by  
Theorem~\ref{the:symmetric_total_surgery_obstruction}~\eqref{the:symmetric_total_surgery_obstruction:vanishing}.
We conclude from Theorem~\ref{the:product_formula} or directly from 
Remark~\ref{the:Special_case_of_Theorem_ref(the:product_formula)} that $s^{\sym}(BG)$ vanishes. 
We conclude from Theorem~\ref{the:existence_of_a_normal_map} that $\caln(BG)$ is not empty. 
Now Theorem~\ref{the:Vanishing_of_the_surgery_obstruction} follows from
Theorem~\ref{the:symmetric_total_surgery_obstruction}~\ref{the:symmetric_total_surgery_obstruction:normal}.
\end{proof}

\begin{proof}[Proof of Theorem~\ref{the:stable_Cannon_Conjecture}]
The considerations above  applied in the special case $N = T^3$
show that $BG \times T^3$ is homotopy equivalent to  a closed topological manifold.

Now let $N$ be any closed smooth manifold, closed PL-manifold, or closed topological manifold respectively of
dimension $\ge 2$.  We conclude from Theorem~\ref{the:Poincare_duality_groups_in_dimension_three}
applied in the case $N_0 = T^3$
that there exists a normal map of degree one for some vector bundle $\xi$ over $BG$
\[
\xymatrix{TM \oplus \underline{\IR^a} \ar[r]^-{\overline{f}} \ar[d]
&
(\xi \times TN) \oplus \underline{\IR^b} \ar[d]
\\
M \ar[r]^-f
&
BG \times N
}
\]
such that $M$ is a smooth manifold, PL-manifold, or topological manifold respectively and
$f$ is a simple homotopy equivalence. 

The considerations above applied in the case, where $N$ is aspherical and 
$n \ge 3$, imply the existence of a  closed topological manifold $M_0$,
together with a homotopy equivalence $f_0 \colon M \to BG \times N$ and a homeomorphism
\[
(U,u) \colon (\overline{\widehat{M_0}}, \partial \overline{\widehat{M_0}}) \to (D^3 \times N, S^2 \times N).
\]
We conclude from Theorem~\ref{the:Borel} that $M_0$ and $M$ are homeomorphic. This finishes the proof of 
Theorem~\ref{the:stable_Cannon_Conjecture}.
\end{proof}


\typeout{-------------------------------------- References  ---------------------r------------------}



\end{document}